\documentclass[a4paper,12pt,oneside]{article}
\usepackage[english]{babel}
\usepackage[T1]{fontenc} 
\usepackage[utf8]{inputenc}
\usepackage{amsthm}
\usepackage{bbm}
\usepackage{amsmath}
\usepackage{amssymb}  
\usepackage{indentfirst}
\usepackage{fancyhdr}
\usepackage{amsthm}
\usepackage{graphicx}
\usepackage{pdfpages}
\usepackage{esint}
\usepackage{url}

\numberwithin{equation}{section}

\theoremstyle{plain} 
\newtheorem{thm}{Theorem}[section] 
 
\newtheorem{lem}[thm]{Lemma} 
\newtheorem{prop}[thm]{Proposition} 
\newtheorem{rmk}[thm]{Remark} 
\newtheorem{dfn}[thm]{Definition}

\usepackage{xcolor}                    
\definecolor{custom-blue}{RGB}{0,99,166} 
\usepackage{hyperref}
\hypersetup{colorlinks=true, allcolors=custom-blue}

\usepackage{geometry}
\allowdisplaybreaks[4]

\begin{document}
\author{$\text{\sc{Stefania Russo}}^*$}

\title{Higher Differentiability of Minimizers for Non-Autonomous Orthotropic Functionals}

\maketitle
\maketitle

\let\thefootnote\relax\footnotetext{
			\small $^{*}$ \textbf{Stefania Russo, }Dipartimento di Matematica e Applicazioni ``R. Caccioppoli'', Università degli Studi di Napoli ``Federico II'', Via Cintia, 80126 Napoli,
 Italy. E-mail: \textit{stefania.russo3@unina.it}}

\begin{abstract}
We establish the higher differentiability for the minimizers of the following non-autonomous integral functionals
\begin{equation*}
 \mathcal{F}(u,\Omega):= \, \displaystyle \Huge  \int_\Omega \sum_{i=1}^{n} \, a_i(x) \lvert u_{x_i} \rvert^{p_i} dx,
\end{equation*}
with  exponents \( p_i \geq 2 \) and with  coefficients \( a_i(x) \) that satisfy a suitable Sobolev regularity. The main result is obtained, as usual, by imposing a gap bound on the exponents \( p_i \), which depends on the dimension and on the degree of regularity of the coefficients \( a_i(x) \).
\end{abstract}


\section{Introduction}
In this paper, we investigate the differentiability properties of local minimizers of convex functionals, exhibiting an orthotropic structure and depending also on the $x-$variable. More precisely, we consider
\begin{equation}\label{functional}
 \mathcal{F}(u,\Omega):= \, \int_\Omega f(x, Du) dx,
\end{equation}
where $\Omega \subset \mathbb{R}^n$ is a bounded open subset, $n\geq 2$, $u : \Omega \to \mathbb{R}$, and $f:\Omega\times\mathbb{R}^{ n}\to [0,+\infty)$ has the following form $$f= f(x, \xi)= \sum_{i=1}^{n} \, a_i(x) \lvert \xi_i \rvert^{p_i}.$$ 
Here $a_i(x)$ are bounded measurable coefficients such that $0< \lambda \leq a_i (x) \leq \Lambda  $ for some constants $\lambda, \Lambda$, for every $i = 1,2,...,n$ and a.e. $ x \in \Omega$.  For the exponents $2\leq p_1 \le p_2 \leq ... \leq p_n$, we shall denote by $\mathbf{p} = (p_1, p_2, \ldots, p_n)$ and by
\begin{equation*}
    W^{1,\mathbf{p}}_{\mathrm{loc}}(\Omega) = \left\{ u \in W^{1,1}_{\mathrm{loc}}(\Omega) : u_{x_i} \in L^{p_i}_{\mathrm{loc}}(\Omega), \, i = 1, \ldots, n \right\},
    \end{equation*}
the  corresponding anisotropic Sobolev space.\\
 One can easily check  that there exist positive constants $ L, l$ depending on $\lambda, \Lambda, p_i$ such that
\begin{equation}\label{A2}
    \langle D_\xi f(x, \xi)-D_\xi f(x, \eta),\xi-\eta\rangle  \geq  l \sum_{i=1}^{n}(|\xi_i|^2+|\eta_i |^2)^\frac{p_i -2}{2}|\xi_i - \eta_i |^2 \tag{A1}
\end{equation}
\begin{equation}\label{A3}
     |D_\xi f(x, \xi)-D_\xi f(x, \eta)| \leq  L \sum_{i=1}^{n}(|\xi_i |^2+|\eta_i |^2)^\frac{p_i -2}{2}|\xi_i - \eta_i | \tag{A2}
\end{equation}

\noindent for a.e. $x \in \Omega$ and all $\xi, \eta \in \mathbb{R}^{ n}$.
Concerning the dependence on the $x$-variable, we assume that $$a_i(x) \in W^{1,r}_{\mathrm{loc}}(\Omega) \qquad \forall i=1, \dots,n ,$$ so that,
denoting by $g(x)=  \underset{i}{\max} \{|Da_i(x)| \}$, we have that $g(x) \in L_{\mathrm{loc}}^r (\Omega)$ and the following inequality
\begin{equation}\label{A4}
    |D_\xi f(x,\xi)-D_\xi f(y, \xi)| \leq |x-y|\,(g(x)+g(y)) \, \sum_{i=1}^{n} |\xi_i|^{p_i-1} \tag{A3}
\end{equation}
\noindent holds for a.e.\ $x,y \in \Omega$ and every $\xi \in \mathbb{R}^{n}$.\\

Let us recall the following definition of local minimizer.
\begin{dfn}
{\em A function $u\in W_{\rm loc}^{1,\mathbf{p}}(\Omega, \mathbb{R}^N)$ is a  local minimizer of
\eqref{functional} if, for every open subset $\tilde{\Omega} \Subset \Omega$, we have   
$\mathcal{F}(u;\tilde{\Omega})\le  \mathcal{F}(\varphi;\tilde{\Omega})$ 
for all $\varphi\in u+W_0^{1,\textbf{p}}(\tilde{\Omega},\mathbb{R}^N)$.}
\end{dfn}
The energy densities $f$ that satisfy anisotropic growth condition fit into the wider context of those with $(p,q)-$growth condition as follows
$$|\xi|^p \leq f(x,\xi) \leq (1 + |\xi|^2)^{\frac{q}{2}}, \qquad1<p<q, $$
where $p= c(p_i)$ and $q= \max \{p_i, i=1, \dots, n \}.$\\
 The regularity theory for functionals with non-standard growth conditions was first developed by P. Marcellini in the pioneering works \cite{MarcelliniEs1, MarcelliniEs2}.
 From the beginning it has been clear  that local minimizers of functionals like \( \mathcal{F} \) may become unbounded when the ratio \( \frac{q}{p} \) is too large.  Indeed, the gap $\frac{q}{p}$ cannot differ to much from $1$ and sufficient conditions to the regularity can be expressed as
$$\frac{q}{p}\le c(n) \to 1 \, \text{as } \ n \to \infty,$$
as proved by the counterexamples in(\cite{Giaquinta}, \cite{MarcelliniEs1}---\cite{MarcelliniEs3}). Since then, a huge number of papers dealing with the $(p,q)-$growth problem have been published (see, for example \cite{Adi,Barone,6}, \\ \cite{12}---\cite{Ele3},\cite{25,Hasto, Koch}). However, since it is impossible to provide an exhaustive list of references, we refer the interested reader to the recent survey (\cite{surveyMar1}---\cite{surveyMing}).\\
However, compared to these papers, the difference with our work lies in the fact that the behavior of our functional at \eqref{functional} depends on the components of the gradient and although they exhibit $(p,q)$-growth, they need to be treated with appropriate arguments.\\
Actually, the regularity of local minimizers of functionals with orthotropic structure is a well-studied problem: for the specific case of  subquadratic nonstandard growth, some results can be found in  \cite{LemmaBrasco}---\cite{6}, \cite{17}, for \( f = 0 \). In the case of a right-hand side \( f \neq 0 \), we refer to  \cite{BraTau, Carozza, 53}.

It has been recently proved in \cite{Brascop12} (sub-quadratic case) and in \cite{LemmaBrasco} (superquadratic case) that any local minimizer \( U \in W^{1,p}_{\text{loc}}(\Omega) \cap L^\infty(\Omega) \) is such that \( DU \in L^\infty_{\text{loc}}(\mathbb{R}^N) \) and \( |U_{x_i}|^{\frac{p_i-2}{2}} U_{x_i} \in W^{1,2}_{\text{loc}}(\Omega) \), for \( i = 1, \dots, N \). One of the main difficulties of functionals with orthotropic structure is that they are degenerate.\\
In \cite{BraTau} they even consider the following widely degenerate functional 
$$F(u; \Omega) = \sum_{i=1}^{N} \int_{\Omega} \frac{1}{p_i} \left( |u_{x_i}| - \delta_i \right)^{p_i} \, dx + \int_{\Omega} f(u) \, dx, $$
with $ u \in W^{1,p}_{\text{loc}}(\Omega)\cap L^\infty(\Omega), \, f \in W^{1,p'}_{\text{loc}}(\Omega), $
and they proved the higher differentiability of the local minimizers
\[
\left( |u_{x_i}| - \delta_i \right)^\frac{p_i}{2}_+ \frac{u_{x_i}}{|u_{x_i}|}  \in W^{1,2}_{\text{loc}}(\Omega), \quad i = 1, \dots, N.
\]
In particular, for local weak solutions of the anisotropic orthotropic $p$-Laplace equation (i.e. the case $\delta_i = 0$), they get
\[
|u_{x_i}|^{\frac{p_i - 2}{2}}  \, u_{x_i} \in W^{1,2}_{\text{loc}}(\Omega), \quad i = 1, \dots, N.
\]
Here we consider the case $\delta_i = 0$, but differently from \cite{LemmaBrasco}-\cite{BraTau}, we allow the presence of some coefficients $a_i(x)$  and we are interested in the conditions that must be imposed on the coefficients in order to achieve analogous higher differentiability results for the solutions. Our motivation comes from the recent paper by P. Marcellini \cite{16}, in which is clearly stated that currently there are few known regularity results for problems of this type, in relation to the properties of the coefficients. Indeed, as far as we know, few regularity results are available for the local minimizers of functionals with an orthotropic structure in the so-called non-autonomous case (see \cite{feopass}). We aim to fill this gap by trying to identify how the regularity of the functions $a_i(x)$ influences that of the minimizer. \\
Actually, our main result consists in proving that a suitable Sobolev regularity of the coefficients $a_i(x)$ transfers into a higher differentiability for the local minimizers of \eqref{functional}.\\
We recall that recent studies have shown that the weak differentiability of the map \( D_\xi f \), whether of integer or fractional order with respect to the \( x \)-variable, is a sufficient condition to achieve higher differentiability (see \cite{cup, Torricelli, 25,40} for the case of Sobolev spaces with integer order and \cite{ 2,15, Grimaldi1, Ipocoana1, Russo} for the fractional one) both in case of standard and non standard growth. \\
In particular, the general requirement for obtaining higher differentiability for local minimizers is a Sobolev-type regularity for the coefficients with an exponent
$r$ that is greater than or equal to the dimension 
$n$. When dealing with bounded minimizers, the situation changes, and sufficient assumptions on the Sobolev regularity of the coefficients can be made independently of the dimension (see \cite{6, Colombo, cup,25}). Our result follows this approach. Even though the order of summability of the coefficients depends on the dimension $n$, it does so only through the bound that allows us to handle locally bounded minimizers.\\
Our aim is to prove the following
\begin{thm}\label{thmBfinito}
 Let $u \in W^{1,\textbf{p}}(\Omega, \mathbb{R}^{n })$ be a local minimizer of \eqref{functional} under assumptions \eqref{A2}--\eqref{A4} with exponents $p_i\geq 2, \forall i=1,\dots,n $ such that
\begin{equation}
p_n <
\begin{cases}
     \min\{ \bar{p}^*,p_1 +2 \} \, &\text{if } p_n < \bar{p}^*,  \\
     p_1 +2 \quad &\text{otherwise},
    \end{cases}
\end{equation}
and with a function $g \in L^{r}_{\mathrm{loc}}(\Omega, \mathbb{R}^n)$, with $r$ such that
\begin{equation}\label{Ipotesir}
   r> p_n + 2 .
\end{equation} 
Then 
\begin{equation}
V_{p_i}(u_{x_i}) \in W^{1,2}_{\mathrm{loc}}(\Omega,\mathbb{R}^{ n}), \qquad\forall i=1,\dots,n
\end{equation}
and the following estimates
\begin{align}
    \sum_{i=1}^{n} \left(  \int_{B_{R/4}}|u_{x_i}|^{p_i +2 } dx \right) & \leq c  \left( \sum_{i=1}^{n} \Vert u_{x_i} \Vert_{L^{p_i}(B_R)} + \Vert g \Vert_{L^r (B_{R})} \right)^\sigma 
\end{align}
and
\begin{align}
      \bar{c}  \sum_{i=1}^{n}  \left( \int_{B_{\frac{R}{4}}} |u_{x_i x_j}|^2|u_{x_i}|^{p_i-2} dx \right) & \leq c  \left( \sum_{i=1}^{n} \Vert u_{x_i} \Vert_{L^{p_i}(B_R)} + \Vert g \Vert_{L^r (B_{R})} \right)^\sigma 
, 
\end{align}
hold for every pair of concentric balls $B_{R/4} \subset B_{R} \Subset \Omega$, where $c = c(n,p_i,\lambda, \Lambda,R)$ and $\sigma= \sigma (n,p_i,q)$ are positive constants.
\end{thm}
Now, we briefly describe the strategy of the proof, which, as usual, will consist of an approximation argument, a uniform a priori estimate, and a passage to the limit. It is well known that the a priori estimate needs to be established for sufficiently regular minimizers. Therefore, the first step is to establish higher differentiability for minimizers of functionals with Lipschitz continuous coefficients, since this result is not available in the literature. More specifically, for a given $\varepsilon \geq 0$, we investigate the function 
\[
f_\varepsilon(x, \xi) = \sum_{i=1}^{n} \tilde{a}_i(x) |\xi_i|^{p_i} + \varepsilon \left( 1 + |\xi|^{\frac{p_n}{2}} \right)^2
\]
where the coefficients $a_i$ satisfy the condition:
\[
L_i = \sup_{x, y \in \Omega, x \neq y} \frac{|\tilde{a}_i(x) - \tilde{a}_i(y)|}{|x - y|} < \infty.
\]
The first goal is to analyze the regularity properties of minimizers of the following functional
\[
\mathcal{F}_\varepsilon(u, \Omega) = \int_\Omega f_\varepsilon(x, Du) \, dx.
\]

After, we aim to establish an uniform a priori estimate, which plays a crucial role in the proof of our main result. Finally the Theorem follows by proving that the a priori estimate is preserved in passing to the limit.\\

We briefly describe the structure of the paper. After summarizing some known results and fixing a few notations, we prove a lemma for approximating functionals to be used in the a priori estimates in Section \ref{PreREG} . Section \ref{Apriori} is devoted to the proof of some a priori estimates for solutions to a family of approximating problems. Next, in Section \ref{Appro}, we pass to the limit in the approximating problems.
\section{Preliminaries}
In this section we introduce some notations and collect several results that we shall use to establish our main result.
We   denote by $c$ or $C$ a general constant that may vary on different occasions, even within the same line of estimates. 

In what follows, $B(x,r)=B_{r}(x)= \{ y \in \mathbb{R}^{n} : |y-x | < r  \}$ will denote the ball centered at $x$ of radius $r$. We will omit the indication of the center $x$ when no confusion arises.

We recall here the well known Sobolev-Troisi inequalities \cite{ feopass, troisi}.
\begin{lem}\label{LemmaTroisi}
    Let $E \subset \mathbb{R}^n$ be a bounded open set and consider $u \in W_0^{1, \textbf{p}}(E), \, p_i \geq 1$ for all $i=1, ..., n,$ and set
    \begin{align}\label{definizionePi}
        \frac{1}{\overline{p}} = \frac{1}{n} \sum_{i=1}^{n} \frac{1}{p_i}, \qquad \overline{p}* = \frac{n \overline{p}}{n- \overline{p}}.
    \end{align}
    If $\overline{p} < n$, then there exists a positive constant $\gamma$ depending only on the  parameters $\{ n, p_1, \dots, p_n \}$, such that
    \begin{align*}
        \lvert \lvert u \rvert \rvert_{\overline{p}*}\leq \gamma \sum_{i=1}^{n} \lvert \lvert u_{x_i} \rvert \rvert_{p_i}.
    \end{align*}
     If $\overline{p} = n$, for every $1 \leq r < \infty$  there exists a positive constant $\gamma_1$ depending only on  $\{ n, r, p_1, \dots, p_n \}$, such that
    \begin{align*}
        \lvert \lvert u \rvert \rvert_{r} \leq \gamma_1  \left[ \sum_{i=1}^{n} \lvert \lvert u_{x_i} \rvert \rvert_{p_i} + \lvert \lvert u \rvert \rvert_{1} \right].
    \end{align*}
      If $\overline{p} > n$,  there exists a positive constant $\gamma_2$ depending only on  $\{ n, p_1, \dots, p_n \}$, such that
    \begin{align*}
        \lvert \lvert u \rvert \rvert_{\infty} \leq \gamma_2  \left[ \sum_{i=1}^{n} \lvert \lvert u_{x_i} \rvert \rvert_{p_i} + \lvert \lvert u \rvert \rvert_{1} \right].
    \end{align*}
\end{lem} 

Now we recall a well-known iteration lemma (see \cite[Lemma $6.1$]{giusti}).
\begin{lem}\label{lm3} 
Let $Z(t)  :  [\rho,R] \rightarrow \mathbb{R}$ be a bounded nonnegative function, where $R>0$. Assume that for $\rho \leq t  < s \leq R$ it holds
$$Z(t) \leq \theta Z(s) +A(s-t)^{- \alpha} + B(s-t)^{- \beta}+ C$$
where $\theta \in (0,1)$, $A$, $B$, $C \geq 0, \alpha> \beta >0$  are constants. Then there exists a constant $c=c(\alpha, \theta, \beta)$ such that
$$Z(\rho) \leq c \biggl( A(R-\rho )^{- \alpha} + B(R-\rho )^{- \beta}+ C\biggr).$$
\end{lem}
At this point, we recall a result that will become significant later (for the proof see  \cite[Proposition $4.3$]{LemmaBrasco}).

\begin{prop}\label{LemmaBrasco}
 Let $u \in W^{1,\textbf{p}}_{\mathrm{loc}}(\Omega)$ and assume that there exists   $k \in \{1,2,\dots,n \}$ such that
    $$|u_{x_k x_k}|^2|u_{x_k}|^{p_k-2} \in L^{1}_{\mathrm{loc}}(\Omega).$$
   If $u \in L^{\infty}_{\mathrm{loc}}(\Omega),$  then $u_{x_k} \in L_{\mathrm{\mathrm{loc}}}^{p_k+2}(\Omega)$  and there exist constants $C=C(n, p_k )>0$ and $\gamma= \gamma(p_k)$ such that for every pair of concentric balls $B_{\rho} \subset B_{R}\Subset  \Omega$,  it holds the following
    \begin{equation}
        \int_{B_\rho} |u_{x_k}|^{p_k+2}  dx \leq C ||u||^2_{\infty} \left(  \int_{B_R} |u_{x_k x_k}|^{2}|u_{x_k}|^{p_k-2}  dx+ \left(\frac{1}{R- \rho} \right)^\gamma \int_{B_R} |u_{x_k}|^{p_k} dx\right).
    \end{equation}
\end{prop}
We shall use the following local boundedness result for minimizers of the  anisotropic functional defined at \eqref{functional}, whose proof can be found in \cite{uLimitato}.
{ \begin{thm}\label{ulimitatoCupini}
    Let the integrand $f = f(x, \xi)$ in \eqref{functional} satisfy the following
\begin{equation}
f(x, \xi) = F(x, |\xi_1|, \ldots, |\xi_i|, \ldots, |\xi_n|), \tag{1.3}
\end{equation}
\[
f(x, \lambda \xi) \leq \lambda^\mu f(x, \xi), \text{ for some } \mu > 1 \text{ and for every } \lambda > 1.
\]
If $p_n < \bar{p}^*$, where $\bar{p}^*$ is the Sobolev exponent of $p$ ($\bar{p}$ defined at \eqref{definizionePi}),
then every local minimizer $u$ of \eqref{functional} is locally bounded.
\end{thm}}
We shall use the  auxiliary function defined by
\begin{center}
$V_{p}(\xi):=(\mu^{2}+|\xi|^{2})^\frac{p-2}{4} \xi $
\end{center}
for all $\xi\in \mathbb{R}^{m}$, $m \in \mathbb{N}$. 
For the function $V_{p}(\xi)$, we recall the following estimate (see e.g. \cite[Lemma 8.3]{giusti}). 
\begin{lem}\label{D1}
Let $1<p<+\infty$. There exists a constants $c_1=c(n,p)>0$ and $c_2=c(n,p)>0$  such that
\begin{center}
$c_1(\mu^{2}+|\xi|^{2}+|\eta|^{2})^{\frac{p-2}{2}} \leq \dfrac{|V_{p}(\xi)-V_{p}(\eta)|^{2}}{|\xi-\eta|^{2}} \leq c_2(\mu^{2}+|\xi|^{2}+|\eta|^{2})^{\frac{p-2}{2}} $
\end{center}
for any $\xi, \eta \in \mathbb{R}^{m}$, $\xi \neq \eta$.
\end{lem}
For further needs we state the following Lemma for the function $V_\delta (\xi)= (\mu^{2}+|\xi|^{2})^\frac{\delta}{2} \xi $. the proof  is obtained by combining \cite[formula 2.1]{Gia} and \cite[Lemma 2.2]{Fusco}.
\begin{lem}\label{Fusco}
Let $\delta \geq 0$. There exists a constant $c_1=c_1(\delta)>0$  such that
  $$ c_1 \langle V_\delta (\xi)- V_\delta (\eta), \xi- \eta \rangle \geq   \, |\xi-\eta|^{2}(\mu^{2}+|\xi|^{2}+|\eta|^{2})^{\frac{\delta}{2}}$$
\end{lem}
We conclude this section with an algebric inequality (see \cite{Duzaar, giusti})
\begin{lem}\label{Duzaar}
  For any $\alpha > 0$, there exists a constant $c = c(\alpha)$ such that, for all $\eta, \zeta \in \mathbb{R}^n \setminus \{0\}$, $n \in \mathbb{N}$, we have
\[
\frac{1}{c} \left||\eta|^{\alpha - 1} \eta - |\zeta|^{\alpha - 1} \zeta  \right| \leq \left( |\eta| + |\zeta| \right)^{\alpha - 1} |\eta - \zeta| \leq c \left| \, |\eta|^{\alpha - 1} \eta - |\zeta|^{\alpha - 1} \zeta \,  \right|.
\]
\end{lem}

\subsection{Difference quotient}
\label{secquo}
In this section we recall some properties of the finite difference quotient operator that will be needed in the sequel. 

\begin{dfn}
Let $F$ be a function defined in an open set $\Omega \subset \mathbb{R}^n$, let $h$ be a real number,  the finite difference operator $\tau_{s,h}F(x)$ is defined as follows
$$ 
\tau_{s,h}F(x) :=F(x+he_s)-F(x) ,$$
where $e_s$ denotes the direction of the $x_s$ axis.
\end{dfn}
The function $\tau_{s,h}F$ is defined in the set
$$\Omega_{|h|}: = \{ x \in \Omega : \mathrm{dist}(x,\partial \Omega)> |h|  \}= \{  x \in \Omega : x+he_s \in \Omega \}.$$
We start with the description of some elementary properties that can be found, for example, in \cite{giusti}. When no confusion can arise, we shall omit the index $s$, and we shall write simply $\tau_{h}$ instead of $ \tau_{s, h}$
\begin{prop}\label{rapportoincrementale}
Let $F \in W^{1,p}(\Omega)$, with $p \geq1$, and let $G:\Omega \rightarrow \mathbb{R}$ be a measurable function.
Then
\\(i) $\tau_{h}F \in W^{1,p}(\Omega_{|h|})$ and 
$$D_{i}(\tau_{h}F)=\tau_{h}(D_{i}F).$$
(ii) If at least one of the functions $F$ or $G$ has support contained in $\Omega_{|h|}$, then
$$\displaystyle\int_{\Omega}F \tau_h G   dx = -\displaystyle\int_{\Omega} G \tau_{-h}F dx.$$
(iii) We have $$\tau_{h} (FG)(x)= F(x+h)\tau_{h} G(x)+G(x) \tau_{h} F(x).$$
\end{prop}
The next result about the finite difference operator is a kind of integral version of Lagrange Theorem.
\begin{lem}\label{ldiff}
If $0<\rho<R,$ $|h|<\frac{R-\rho}{2},$ $1<p<+\infty$ and $F, \, D_s F\in L^{p}(B_{R})$, then
\begin{center}
$\displaystyle\int_{B_{\rho}} |\tau_{s, h}F(x)|^{p} dx \leq c(n,p)|h|^{p} \displaystyle\int_{B_{R}} |D_sF(x)|^{p} dx$.
\end{center}
Moreover,
\begin{center}
$\displaystyle\int_{B_{\rho}} |F(x+h)|^{p} d x \leq  \displaystyle\int_{B_{R}} |F(x)|^{p}d x$.
\end{center}
\end{lem}
\noindent We now state a lemma that will be needed in the sequel.
\begin{lem}\label{Dd}
Let $F \in L^2(B_R)$. Suppose that there exist $\rho \in (0,R)$, and $M >0$ such that $$  \displaystyle\int_{B_{\rho}}|\tau_{h}F(x)|^2 dx \leq M^2 |h|^{2 }, $$ for every $h$ such that $|h| < \frac{R-\rho}{2}$. Then $F \in L^{\frac{2n}{n-2 }}(B_{\rho})$ and $$\Vert  F \Vert _{L^{\frac{2n}{n-2 }}(B_{\rho})} \leq c (M+ \Vert F \Vert _{L^2(B_R)}),  $$ with $c=c(n,R,\rho)$.
\end{lem}

We conclude by recalling the following
\begin{lem}\label{Lemmahzero}
If $0<\rho<R,$ $|h|<\frac{R-\rho}{2},$ $1<p<+\infty$ and $F\in L^{p}(B_{R})$. If there exists a positive constant $C$ such that
\begin{align}
\displaystyle\int_{B_{\rho}} |\tau_{s, h}F(x)|^{p} dx \leq C|h|^{p}, 
\end{align}
for every $h$, then the distributional derivative $D_s F$ belongs to $L^p(B_\rho).$\\
Moreover it holds
\begin{align}
\displaystyle\int_{B_{\rho}} |DF(x)|^{p} dx \leq C .
\end{align}
\end{lem}

\section{A preliminary regularity result}\label{PreREG}
In this section we prove a regularity result for more regular problems with respect to \eqref{functional} that will be needed in the approximation procedure. More precisely, for $\varepsilon \geq 0$ we consider 
\begin{equation}\label{Eqconepsilon}
 f_\varepsilon(x, \xi)= \sum_{i=1}^{n} \, \tilde{a}_i(x) \lvert \xi_i \rvert^{p_i} + \varepsilon(1+|\xi|^\frac{p_n}{2})^2
\end{equation}
with  Lipschitz continuous coefficients,
i.e.
\begin{equation}\label{atilde}
L_i= \sup_{x,y \in \Omega, x \neq y} \frac{|\tilde{a_i}(x)-\tilde{a}_i(y)|}{|x-y|}< \infty.
\end{equation}
Setting $\tilde{f}(x, \xi)= \sum_{i=1}^{n} \, \tilde{a}_i(x) \lvert \xi_i \rvert^{p_i},$ one can easily check that
\begin{equation}\label{A'4}
    |D_\xi  \tilde{f}(x,\xi)-D_\xi \tilde{f}(y, \xi)| \leq K |x-y|\, \, \sum_{i=1}^{n} |\xi_i|^{p_i-1} \tag{A3'}
\end{equation}
\noindent for a.e.\ $x,y \in \Omega $ and every $\xi \in \mathbb{R}^{ n}$, where $K=K(L_i, p_i)$
.\\
We recall a Lipschitz regularity Theorem for the minimizers of the functional $$ \mathcal{F}_\varepsilon (x, Dv) =\int_\Omega f_\varepsilon(x, Dv),$$ that can be deduced by  \cite[Theorem $2.2$]{GriMasPass} in the case $p=q=p_n$. 
\begin{thm}\label{thm2Mascolo}
    Let $\tilde{f} $ satisfy \eqref{A2}--\eqref{A'4}
    and let $u_\varepsilon \in W^{1,p_n}(\Omega)$ be a local minimizer of \begin{equation}\label{functional2}
 \mathcal{F}_\varepsilon (u,\Omega):= \, \int_\Omega f_\varepsilon(x, Du) dx.
\end{equation} 
Then $u_\varepsilon \in W^{1,\infty}_{\mathrm{loc}}(\Omega)$ and the following estimate 
    \begin{equation}\label{lipestimate}
      \Vert D u_\varepsilon \Vert_{L^\infty(B_\rho)} \le C \left[1+ \int_{B_R} f_\varepsilon(x,D u_\varepsilon)dx  \right]^\frac{\sigma}{p_1},  
    \end{equation}
     holds for every pair of concentric balls $B_\rho \subset B_R \Subset \Omega$, with $C$ and $\sigma$ positive constants depending on $n,p_n,\rho$, $\varepsilon$ and $R$.\\
\end{thm}
We'd like to notice that, 
     since the function $f_\varepsilon$ satisfies standard $p_n$-growth, in Theorem \ref{thm2Mascolo} we do not need to require that $$\frac{p_n}{p_1}< 1+ \frac{1}{n},$$
     as done in \cite{GriMasPass}.\\
\noindent We shall use the higher differentiability of the minimizers $u_\varepsilon$ of $\mathcal{F}_\varepsilon$ that, as far  we know, is not available in literature and that is contained in the following
{
\begin{lem}\label{LemmaInduzione}
  Let $f$ satisfy \eqref{A2}, \eqref{A3} and \eqref{A'4} for exponents $ p_i \geq 2,  \forall i=1,\dots,n$ and 
    let  $u_\varepsilon \in W^{1,p_n}_{\mathrm{loc}}(\Omega) $ be a local minimizer of the functional \eqref{functional2}.
    Then 
    we have  $$(V_{p_i}((u_\varepsilon)_{x_j})) \in W^{1,2}_{\mathrm{loc}}(\Omega), \quad \forall i \in \{1,\dots, n\} \text{ and  } \, \forall i \in \{1,\dots, n\}.$$
\end{lem}
\begin{proof}
We start by observing that $u_\varepsilon$
 solves the following Euler-Lagrange system
\begin{equation}\label{System}
    \int_{\Omega} \langle f_{\xi} (x, D u_\varepsilon(x))+ \varepsilon(1+ |D u_\varepsilon|^2)^{\frac{p_n-2}{2}}D u_\varepsilon, D \varphi \rangle \, dx = 0, 
\end{equation}
for every $ \varphi \in W^{1,p_n}_{0}(\Omega, \mathbb{R}^n).$

Fix a ball $B_R \Subset \Omega$ and consider radii $\rho <t< R$, a cut-off function $\eta \in C_0^{\infty}(B_t)$, with $\eta=1$ on $B_{\rho}$, $0 \leq \eta \leq 1$, $|D \eta | \leq \frac{c}{t-\rho}$  and $|h|\leq \frac{R - t}{2}$. We test \eqref{System} with the function
$$\varphi = \tau_{j, -h} \left(  \eta^2 \tau_{j,h} (u_\varepsilon) \right) $$
thus obtaining
\begin{equation*}
    \int_{\Omega} \langle f_{\xi} (x, D u_\varepsilon(x))+ \varepsilon(1+ |D u_\varepsilon|^2)^{\frac{p_n-2}{2}}D u_\varepsilon, \tau_{j,-h} D  \left(  \eta^2 \tau_{j,h} (u_\varepsilon) \right) \rangle \, dx = 0 \, , 
\end{equation*}
and hence, by $(ii)$ of Proposition \ref{rapportoincrementale},
\begin{equation}\label{IntEq}
    \int_{\Omega} \langle \tau_{j,h} f_{\xi} (x, D u_\varepsilon(x))+ \varepsilon\tau_{j,h}(1+ |D u_\varepsilon|^2)^{\frac{p_n-2}{2}}D u_\varepsilon, D  \left(  \eta^2 \tau_{j,h} (u_\varepsilon) \right) \rangle \, dx = 0 \, .
\end{equation}
Since
$$D(\eta^2 \tau_{j,h} (u_\varepsilon)) = 2 \eta D \eta \tau_{j,h} (u_\varepsilon) + \eta^2 \tau_{j,h} D u_\varepsilon$$
and
\begin{align*}
\tau_{j,h} f_{\xi} (x, D u_\varepsilon(x))+ \varepsilon\tau_{j,h}(1+ |D u_\varepsilon|^2)^{\frac{p_n-2}{2}}D u_\varepsilon &=   f_\xi (x+e_jh, D u_\varepsilon(x+e_jh))- f_\xi (x+e_jh, D u_\varepsilon(x)) \\
& \qquad + f_\xi (x+e_jh, D u_\varepsilon(x))- f_\xi (x, D u_\varepsilon(x))\\
&  \qquad + \varepsilon\tau_{j, h}(1+ |D u_\varepsilon(x)|^2)^\frac{p_n-2}{2}D u_\varepsilon(x),
\end{align*}
we  may rewrite equality \eqref{IntEq} as follows
\begin{align}\label{s}
 0=\int_{\Omega} &\langle f_{\xi} (x+e_jh, D u_\varepsilon(x+e_jh))- f_{\xi} (x+e_jh, D u_\varepsilon(x)) , \eta^2  \tau_{j,h} D u_\varepsilon \rangle \, dx \notag \\
    &+ 2 \int_{\Omega} \langle f_{\xi} (x+e_jh, D u_\varepsilon(x+e_jh))- f_{\xi} (x+e_jh, D u_\varepsilon(x)) ,  \eta D \eta \tau_{j,h} (u_\varepsilon) \rangle \, dx \notag \\
    &+  \int_{\Omega} \langle f_{\xi} (x+e_jh, D u_\varepsilon(x))- f_{\xi} (x, D u_\varepsilon(x)) , 2 \eta D \eta \tau_{j,h} (u_\varepsilon) \rangle \, dx \notag \\
    &+  \int_{\Omega} \langle f_{\xi} (x+e_jh, D u_\varepsilon(x))- f_{\xi} (x, D u_\varepsilon(x)) , \eta^2  \tau_{j,h} D u_\varepsilon \rangle \, dx \notag \\
    &+ 2 \varepsilon  \int_{\Omega} \langle \tau_{j, h}(1+ |D u_\varepsilon(x)|^2)^\frac{p_n-2}{2}D u_\varepsilon(x),  \eta D \eta \tau_{j,h} (u_\varepsilon) \rangle \, dx  \notag \\
     &+ \varepsilon  \int_{\Omega} \langle \tau_{j, h}(1+ |D u_\varepsilon(x)|^2)^\frac{p_n-2}{2}D u_\varepsilon(x),  \eta^2  \tau_{j,h} D u_\varepsilon \rangle \, dx \notag \\
    & =: J_1 + J_2 + J_3 + J_4 + J_5+ J_6.
\end{align}
Therefore
\begin{equation}\label{SommaJ}
   J_1 +J_6 \leq |J_2|  + |J_3| + |J_4| +|J_5|  .
\end{equation}
 From \eqref{A2}, we infer
\begin{align}\label{J_1}
    J_1 &=  \int_{\Omega} \langle f_{\xi} (x+e_j h, D u_\varepsilon(x+ e_jh))- f_{\xi} (x+ e_j h, D u_\varepsilon(x)) , \eta^2  \tau_{j,h} D u_\varepsilon \rangle \, dx \notag \\
    & \geq l \int_{\Omega} \eta^2 \sum_{i=1}^{n}\, \left( \, |(u_\varepsilon)_{x_i}(x+e_jh)|^2+|(u_\varepsilon)_{x_i}(x) |^2 \, \right)^\frac{p_i -2}{2}| \tau_{j,h} (u_\varepsilon)_{x_i} |^2 dx =: l \textbf{\textbf{RHS}}
\end{align}
The integral $J_6$ can be estimated by using Lemma \ref{Fusco} for $\delta= p_n-2$, as follows
\begin{equation}\label{J6}
    J_6 \geq \frac{\varepsilon}{c} \int_{\Omega} \eta^2 (1+ |D u_\varepsilon(x+ he_j)|^2 + |D u_\varepsilon(x)|^2)^\frac{p_n-2}{2} |\tau_{j, h}( D u_\varepsilon(x))|^2 =: \frac{\varepsilon}{c} \overline{\textbf{\textbf{RHS}}},
\end{equation}
where $c=c(p_n)$ is the constant appearing in Lemma \ref{Fusco}.\\
Now we consider the term $|J_2|$. By hypotesis \eqref{A3} and by  Young's  inequality, we get
\begin{align}\label{J_2}
    |J_2| &= 2\left|   \int_{\Omega} \langle f_{\xi} (x+e_jh, D u_\varepsilon(x+e_jh))- f_{\xi} (x+e_jh, D u_\varepsilon(x)) ,  \eta \,   \eta_{x_i} \tau_{j,h} (u_\varepsilon) \rangle \, dx  \right| \notag \\
    & \leq 2 c \int_{\Omega} \eta \sum_{i=1}^{n}\, \left( \, |(u_\varepsilon)_{x_i}(x+e_jh)|^2+|(u_\varepsilon)_{x_i}(x) |^2 \, \right)^\frac{p_i -2}{2}| \tau_{j,h} (u_\varepsilon)_{x_i} | | \tau_{j,h} (u_\varepsilon) ||\eta_{x_i}| dx  \notag \\
& \leq \beta \ \ \ \ \underbrace{\int_{\Omega}  \eta^2 \sum_{i=1}^{n}\, \left( \, |(u_\varepsilon)_{x_i}(x+e_jh)|^2+|(u_\varepsilon)_{x_i}(x) |^2 \, \right)^\frac{p_i -2}{2}| \tau_{j,h} (u_\varepsilon)_{x_i} |^2 dx}_{\textbf{\textbf{RHS}}}  \notag \\
& \qquad + c_\beta \int_{\Omega} \sum_{i=1}^{n}\, \left( \, |(u_\varepsilon)_{x_i}(x+e_jh)|^2+|(u_\varepsilon)_{x_i}(x) |^2 \, \right)^\frac{p_i -2}{2}  | \tau_{j,h} (u_\varepsilon) |^2 \, |\eta_{x_i}|^2 dx  \notag \\
& \leq  \beta \,\textbf{\textbf{RHS}} \,  + \frac{c_\beta}{(t-\rho)^2}  \int_{B_t} \sum_{i=1}^{n}\,  \, |(u_\varepsilon)_{x_i}(x+ e_jh)+(u_\varepsilon)_{x_i}(x)|^{p_i -2}  | \tau_{j,h} (u_\varepsilon) |^2 dx \notag \\
& \leq \beta \,\textbf{\textbf{RHS}} \,  + \frac{c_\beta(R)}{(t-\rho)^2} |h|^2 \sum_{i=1}^{n} \lvert \lvert  (u_\varepsilon)_{x_i} \rvert \rvert_{L^\infty (B_R)}^{p_i-2} \left(\int_{B_{R}}   \, |(u_\varepsilon)_{x_j}(x)|^{p_j}dx \right)^{\frac{2}{p_j}},
\end{align}
where to estimate the last integral in the right-hand side  we used Theorem \ref{thm2Mascolo}, H\"older's inequality and  Lemma \ref{ldiff}  .\\

Now, we turn our attention to $|J_3|$. From condition \eqref{A'4} and the properties of $\eta$, we get
\begin{align}\label{J_3}
    |J_3|&= \left| \int_{\Omega} \langle f_{\xi} (x+e_j h, D u_\varepsilon(x))- f_{\xi} (x, D u_\varepsilon(x)) , 2 \eta D \eta \tau_{j,h} (u_\varepsilon) \rangle \, dx  \right| \notag \\
    & \leq |h|\, K\,  \sum_{i=1}^n \int_{\Omega} 2\eta | \eta_{x_i}||\tau_{j, h} (u_\varepsilon)| \,|(u_\varepsilon)_{x_i}|^{p_i-1} dx \notag \\
   & \leq |h| \frac{c(K)}{t-\rho}\, \sum_{i=1}^n  \, \left(   \int_{B_t}  |\tau_{j,h}(u_\varepsilon)|\,|(u_\varepsilon)_{x_i}|^{p_i-1} dx \right) \notag \\
   & \leq |h|^2 \frac{c(K, R)}{t-\rho}\,  \sum_{i=1}^{n} \lvert \lvert  (u_\varepsilon)_{x_i} \rvert \rvert_{L^\infty(B_R)}^{p_i-1} \left(\int_{B_{R}}   \, |(u_\varepsilon)_{x_j}(x)|^{p_j}dx \right)^{\frac{1}{p_j}},
\end{align}
where, as before, we used Theorem \ref{thm2Mascolo}, H\"older's inequality and  Lemma \ref{ldiff}  .\\
Similarly as for the estimate of $J_3$, from condition \eqref{A'4} and the properties of $\eta$, we obtain
\begin{align}\label{J_4}
     |J_4| &= \left|  \int_{\Omega} \langle f_{\xi} (x+e_jh, D u_\varepsilon(x))- f_{\xi} (x, D u_\varepsilon(x)) , \eta^2  \tau_{j,h} D u_\varepsilon \rangle \, dx \right| \notag \\
     &= \left|  \int_{\Omega} \sum_{i=1}^{n} \eta^2 \Big( a_i(x+e_jh)- a_i(x) \Big)|(u_\varepsilon)_{x_i}|^{p_i -2}\, (u_\varepsilon)_{x_i} \, \tau_{j, h}(u_\varepsilon)_{x_i} dx\right| \notag \\
    & \leq c |h|\, K\, \sum_{i=1}^{n}  \int_{\Omega} \eta^2 |\tau_{j, h}(u_\varepsilon)_{x_i}| \, |(u_\varepsilon)_{x_i}|^{p_i -1} dx \notag \\
    & \leq \beta \ \ \ \ \underbrace{\int_{\Omega} \eta^2 \sum_{i=1}^{n}\, \left( \, |(u_\varepsilon)_{x_i}(x+e_j h)|^2+|(u_\varepsilon)_{x_i}(x) |^2 \, \right)^\frac{p_i -2}{2}| \tau_{j,h} (u_\varepsilon)_{x_i} |^2 dx}_{\textbf{\textbf{RHS}}}  \notag \\
    & \qquad + c_{\beta} |h|^2 \, \sum_{i=1}^{n}  \int_{B_t} \eta^2 |(u_\varepsilon)_{x_i}|^{p_i} \,  dx \notag \\
     & \leq \beta \textbf{\textbf{RHS}} + c_{\beta} |h|^{2} \,\sum_{i=1}^{n}   \int_{B_t}  |(u_\varepsilon)_{x_i}|^{p_i} \,  dx .
\end{align}
Now, we take care of $|J_5|$. From Lemma \ref{D1}, Young's and H\"older's inequalities with exponents $\frac{p_n}{2}$ and $\frac{p_n}{p_n-2}$ and the properties of $\eta$, we get
\begin{align}\label{J5}
|J_5| &\leq 2 \varepsilon  \int_{\Omega} |\tau_{j, h}(1+ |D u_\varepsilon(x)|^2)^\frac{p_n-2}{2} D u_\varepsilon(x))| \,  |\eta| |D \eta| |\tau_{j,h} (u_\varepsilon)|  \, dx \notag \\ 
&\leq c \varepsilon \int_{\Omega} (1+ |D u_\varepsilon(x+ he_j)|^2 + |D u_\varepsilon(x)|^2)^\frac{p_n-2}{2} |\tau_{j, h}( D u_\varepsilon(x))| \,  \eta |D \eta| |\tau_{j,h} (u_\varepsilon)|  \, dx \notag \\
&\leq \frac{\varepsilon}{2c} \underbrace{   \int_{\Omega} \eta^2 (1+ |D u_\varepsilon(x+ he_j)|^2 + |D u_\varepsilon(x)|^2)^\frac{p_n-2}{2} |\tau_{j, h}( D u_\varepsilon(x))|^2  }_{= \overline{\textbf{\textbf{RHS}}}} \, dx \notag \\
& \qquad + \frac{c \varepsilon}{(t-\rho)^2} \int_{B_t} (1+ |D u_\varepsilon(x+ he_j)|^2 + |D u_\varepsilon(x)|^2)^\frac{p_n-2}{2}   |\tau_{j,h} (u_\varepsilon)|^2  \, dx \notag \\ 
& \leq \frac{\varepsilon}{2c}\overline{\textbf{\textbf{RHS}}}\ +  \frac{c \varepsilon}{(t-\rho)^2} \left( \int_{B_t} (1+ |D u_\varepsilon(x+ he_j)| + |D u_\varepsilon(x)|)^{p_n} dx \right)^{\frac{p_n-2}{p_n}} \left( \int_{B_t}  |\tau_{j,h} (u_\varepsilon)|^{p_n} dx \right)^{\frac{2}{p_n}} \, \notag \\ 
& \leq \frac{\varepsilon}{2c} \overline{\textbf{\textbf{RHS}}}\ +  \frac{c \varepsilon}{(t-\rho)^2} |h|^2 \left( \int_{B_{R}} (1+|D u_\varepsilon(x)|)^{p_n} dx \right)^{\frac{p_n-2}{p_n}} \left( \int_{B_t}  |D u_\varepsilon|^{p_n} dx \right)^{\frac{2}{p_n}} \, \notag \\ 
& \leq \frac{\varepsilon}{2c} \overline{\textbf{RHS}}\ +  \frac{c \varepsilon}{(t-\rho)^2} |h|^2 \int_{B_{R}} (1+|D u_\varepsilon(x)|)^{p_n} dx ,
\end{align}
where we used Lemma \ref{ldiff}.\\
Inserting \eqref{J_1}, \eqref{J6}, \eqref{J_2}, \eqref{J_3}, \eqref{J_4}, \eqref{J5}  in \eqref{SommaJ} 
and reabsorbing the term $\overline{\textbf{RHS}}$
, we obtain
\begin{align}\label{sum}
 l &\int_{\Omega} \eta^2 \sum_{i=1}^{n}\, \left( \, |(u_\varepsilon)_{x_i}(x+e_jh)|^2+|(u_\varepsilon)_{x_i}(x) |^2 \, \right)^\frac{p_i -2}{2}| \tau_{j,h} (u_\varepsilon)_{x_i} |^2 dx + \frac{\varepsilon}{2c} \overline{\textbf{\textbf{RHS}}} \notag \\
& \qquad \leq  2 \beta \, \int_{\Omega}  \eta^2 \sum_{i=1}^{n}\, \left( \, |(u_\varepsilon)_{x_i}(x+e_jh)|^2+|(u_\varepsilon)_{x_i}(x) |^2 \, \right)^\frac{p_i -2}{2}| \tau_{j,h} (u_\varepsilon)_{x_i} |^2 dx  \notag \\ 
& \qquad +  \frac{c_\beta}{(t-\rho)^2} |h|^2 \sum_{i=1}^{n} \lvert \lvert  u_{x_i} \rvert \rvert_{L^\infty(B_R)}^{p_i-2} \left(\int_{B_{R}}   \, |(u_\varepsilon)_{x_j}(x)|^{p_j}dx \right)^{\frac{2}{p_j}} \notag \\
& \qquad +\frac{c(K)}{(t-\rho)}|h|^2 \,  \sum_{i=1}^{n} \lvert \lvert  u_{x_i} \rvert \rvert_{L^\infty(B_R)}^{p_i-1} \left(\int_{B_{R}}   \, |(u_\varepsilon)_{x_j}(x)|^{p_j}dx \right)^{\frac{1}{p_j}} \notag \\
& \qquad + c_{\beta} |h|^{2} \,\sum_{i=1}^{n}   \int_{B_t}  |(u_\varepsilon)_{x_i}|^{p_i} \,  dx \notag \\
& \qquad  +  \frac{c \varepsilon}{(t-\rho)^2} |h|^2 \int_{B_t} (1+|D u_\varepsilon(x)|)^{p_n} dx. 
\end{align}
Since $\overline{\textbf{RHS}}$ is non-negative, we obviously have
\begin{align}\label{RHS2}
      \frac{l}{2} &\int_{\Omega} \eta^2 \sum_{i=1}^{n}\, \left( \, |(u_\varepsilon)_{x_i}(x+e_jh)|^2+|(u_\varepsilon)_{x_i}(x) |^2 \, \right)^\frac{p_i -2}{2}| \tau_{j,h} (u_\varepsilon)_{x_i} |^2 dx \notag \\
 & \leq  \frac{l}{2}\int_{\Omega} \eta^2 \sum_{i=1}^{n}\, \left( \, |(u_\varepsilon)_{x_i}(x+e_jh)|^2+|(u_\varepsilon)_{x_i}(x) |^2 \, \right)^\frac{p_i -2}{2}| \tau_{j,h} (u_\varepsilon)_{x_i} |^2 dx + \frac{\varepsilon}{2c} \overline{\textbf{\textbf{RHS}}}.
\end{align}
Hence, from \eqref{sum} and \eqref{RHS2}, choosing $\beta = \frac{l}{4}$ and reabsorbing the first  term in the right-hand side  by the left-hand side, we get
\begin{align}\label{stimaUnioneJ}
   \frac{l}{2} &\int_{\Omega} \eta^2 \sum_{i=1}^{n}\, \left( \, |(u_\varepsilon)_{x_i}(x+e_jh)|^2+|(u_\varepsilon)_{x_i}(x) |^2 \, \right)^\frac{p_i -2}{2}| \tau_{j,h} (u_\varepsilon)_{x_i} |^2 dx \notag \\
& \leq  \frac{c}{(t-\rho)^2} |h|^2 \sum_{i=1}^{n} \lvert \lvert  (u_\varepsilon)_{x_i} \rvert \rvert_{L^\infty(B_R)}^{p_i-2} \left(\int_{B_{R}}   \, |(u_\varepsilon)_{x_j}(x)|^{p_j}dx \right)^{\frac{2}{p_j}} \notag \\
& \qquad +\frac{c}{(t-\rho)}|h|^2 \, \sum_{i=1}^{n} \lvert \lvert  (u_\varepsilon)_{x_i} \rvert \rvert_{L^\infty(B_R)}^{p_i-1} \left(\int_{B_{R}}   \, |(u_\varepsilon)_{x_j}(x)|^{p_j}dx \right)^{\frac{1}{p_j}} \notag \\
& \qquad + c |h|^{2} \,\sum_{i=1}^{n}   \int_{B_t}  |(u_\varepsilon)_{x_i}|^{p_i} \,  dx  +  \frac{c \varepsilon}{(t-\rho)^2} |h|^2 \int_{B_t} (1+|D u_\varepsilon(x)|)^{p_n} dx,
\end{align}
with a constant $c=c(l,K, R)$.\\
Applying Lemma \ref{D1} in the left-hand side of \eqref{stimaUnioneJ} and using that $\eta=1$ on $B_\rho$, we get 

\begin{align}
 \sum_{i=1}^n \int_{B_{\rho}} |\tau_{j,h} V_{p_i} ((u_\varepsilon)_{x_i})|^2 dx & \leq \,|h|^2 \Biggl\{    \frac{c}{(t-\rho)^2}  \sum_{i=1}^{n} \lvert \lvert  (u_\varepsilon)_{x_i} \rvert \rvert_{L^\infty(B_R)}^{p_i-2} \left(\int_{B_{R}}   \, |(u_\varepsilon)_{x_j}(x)|^{p_j}dx \right)^{\frac{2}{p_j}} \notag \\
& \qquad +\frac{c}{(t-\rho)} \, \sum_{i=1}^{n} \lvert \lvert  (u_\varepsilon)_{x_i} \rvert \rvert_{L^\infty(B_R)}^{p_i-1} \left(\int_{B_{R}}   \, |(u_\varepsilon)_{x_j}(x)|^{p_j}dx \right)^{\frac{1}{p_j}} \notag \\
& \qquad + c  \,\sum_{i=1}^{n}   \int_{B_t}  |(u_\varepsilon)_{x_i}|^{p_i} \,  dx \notag \\
& \qquad  +  \frac{c \varepsilon}{(t-\rho)^2}  \int_{B_{R}} (1+|D u_\varepsilon(x)|)^{p_n} dx \Biggr\}. \notag 
  \end{align}
 and so
\begin{align}
  \sum_{i=1}^n \int_{B_{\rho}} |\tau_{j,h} V_{p_i} ((u_\varepsilon)_{x_i})|^2 dx \leq  c |h|^{2} \,\left[ \sum_{i=1}^{n} \Vert (u_\varepsilon)_{x_i} \Vert_{L^{p_i}(B_{R})}+ \varepsilon \int_{B_{R}} (1+|D u_\varepsilon(x)|)^{p_n} dx \right]^\gamma, \label{StimaLemma}
\end{align}
for positive constants $c = c(n,p_i,\rho, R, \lambda, K, \lvert \lvert Du \rvert \rvert_{\infty})$.
Then we have that 
 $$V_{p_i}((u_\varepsilon)_{x_i}) \in W^{1,2}_{\mathrm{loc}}(\Omega), \qquad \forall i=1, \dots, n$$ 
 which implies in particular that $$|(u_\varepsilon)_{x_i x_j}|^2|(u_\varepsilon)_{x_i}|^{p_j-2} \in L^1_{\mathrm{loc}}(\Omega,\mathbb{R}^{ n}), \qquad \forall i=1, \dots, n \text{ and  }  \forall j=1, \dots, n.$$
\end{proof}
\begin{rmk}
For further needs, we record that, by combining Theorem \ref{ulimitatoCupini} with Proposition \ref{LemmaBrasco}, we have, as a consequence,
$$(u_\varepsilon)_{x_j} \in L^{p_j+2}_{\mathrm{loc}}(\Omega ), \qquad \forall j=1, \dots, n.$$ 
\end{rmk}
}
\section{A priori estimates}\label{Apriori}
The main aim of this section is to establish the following a priori estimate which is the main step in the proof of our main result.
\begin{thm}\label{AppThm}
 Let $u \in W^{1,\textbf{p}}(\Omega, \mathbb{R}^{N })$ be a local minimizer of \eqref{functional} under assumptions \eqref{A2}--\eqref{A4} with exponents $p_i\geq 2, \forall i=1,\dots,n $ such that
\begin{equation}\label{Ipotesip}
p_n <
\begin{cases}
     \min\{ \bar{p}^*,p_1 +2 \} \quad & \text{if } p_n < \bar{p}^*,  \\
     p_1 +2 \qquad & \text{otherwise,}
    \end{cases}
\end{equation}
and with a function $g \in L^{r}_{\mathrm{loc}}(\Omega, \mathbb{R}^N)$, with $r$ such that
\begin{equation}\label{Ipotesir}
   r> p_n + 2 .
\end{equation} 
If we assume that 
\begin{equation}\label{Apriori}
V_{p_i}(u_{x_i}) \in W^{1,2}_{\mathrm{loc}}(\Omega,\mathbb{R}^{ n}), \qquad\forall i=1,\dots,n
\end{equation}
then the following estimates
\begin{align}\label{uxistima}
    \sum_{i=1}^{n} \left(  \int_{B_{R/4}}|u_{x_i}|^{p_i +2 } dx \right) & \leq c  \left( \sum_{i=1}^{n} \Vert u_{x_i} \Vert_{L^{p_i}(B_R)} + \Vert g \Vert_{L^r (B_{R})} \right)^\sigma 
\end{align}
and
\begin{align}
      \bar{c}  \sum_{i=1}^{n}  \left( \int_{B_{\frac{R}{4}}} |u_{x_i x_j}|^2|u_{x_i}|^{p_i-2} dx \right) & \leq c  \left( \sum_{i=1}^{n} \Vert u_{x_i} \Vert_{L^{p_i}(B_R)} + \Vert g \Vert_{L^r (B_{R})} \right)^\sigma 
, \label{StimaTeo1}
\end{align}
hold for every pair of concentric balls $B_{R/4} \subset B_{R} \Subset \Omega$, where $c = c(n,p_i,\lambda, \Lambda,R)$ and $\sigma= \sigma (n,p_i,q)$ are positive constants.
\end{thm}
\begin{proof}
Fix a ball $B_R \Subset \Omega$ and consider radii $\frac{R}{4}< \rho< s<t<t'<r<R$, a cut-off function $\eta \in C_0^{\infty}(B_t)$, with $\eta=1$ on $B_{s}$, $0 \leq \eta \leq 1$, $|D \eta | \leq \frac{c}{t-s}$  and $|h|\leq \frac{t' - t}{2}$.
We test the Euler-Lagrange equation of \eqref{functional} with the function $$\varphi = \tau_{j,-h}(\eta^2 \tau_{j,h}u),$$ 
thus getting
\begin{align}\label{SommaIntera}
 0= \int_{\Omega} &\langle f_{\xi} (x+e_jh, Du(x+e_jh))- f_{\xi} (x+e_jh, Du(x)) , \eta^2  \tau_{j,h} Du \rangle \, dx \notag \\
    &+ \int_{\Omega} \langle f_{\xi} (x+e_jh, Du(x+e_jh))- f_{\xi} (x+e_jh, Du(x)) , 2 \eta D \eta \tau_{j,h} u \rangle \, dx \notag \\
    &+  \int_{\Omega} \langle f_{\xi} (x+e_jh, Du(x))- f_{\xi} (x, Du(x)) , 2 \eta D \eta \tau_{j,h} u \rangle \, dx \notag \\
    &+  \int_{\Omega} \langle f_{\xi} (x+e_jh, Du(x))- f_{\xi} (x, Du(x)) , \eta^2  \tau_{j,h} Du \rangle \, dx \notag \\
    & =: I_1 + I_2 + I_3 + I_4. \notag
\end{align}
From the previous equality, we deduce that 
\begin{equation}\label{SommaInt}
   I_1 \leq |I_2|  + |I_3| + |I_4| .
\end{equation}
By virtue of assumption \eqref{A2}, we infer
\begin{align}\label{I1}
    I_1 &=  \int_{\Omega} \langle f_{\xi} (x+e_j h, Du(x+ e_jh))- f_{\xi} (x+ e_j h, Du(x)) , \eta^2  \tau_{j,h} Du \rangle \, dx \notag \\
    & \geq l \int_{\Omega} \eta^2 \sum_{i=1}^{n}\, \left( \, |u_{x_i}(x+e_jh)|^2+|u_{x_i}(x) |^2 \, \right)^\frac{p_i -2}{2}| \tau_{j,h} u_{x_i} |^2 dx =: l  \textbf{\textbf{RHS}}.
\end{align}

Now we consider the term $|I_2|$. By using hypotesis \eqref{A3}, Young's  inequality, we get
\begin{align}\label{termtau}
    |I_2| &= \left|   \int_{\Omega} \langle f_{\xi} (x+e_jh, Du(x+e_jh))- f_{\xi} (x+e_jh, Du(x)) , 2 \eta \,   \eta_{x_i} \tau_{j,h} u \rangle \, dx  \right| \notag \\
    & \leq 2 L \int_{\Omega} \eta \sum_{i=1}^{n}\, \left( \, |u_{x_i}(x+e_jh)|^2+|u_{x_i}(x) |^2 \, \right)^\frac{p_i -2}{2}| \tau_{j,h} u_{x_i} | | \tau_{j,h} u ||\eta_{x_i}| dx  \notag \\
& \leq 2 \varepsilon \ \ \ \ \underbrace{\int_{\Omega}  \eta^2 \sum_{i=1}^{n}\, \left( \, |u_{x_i}(x+e_jh)|^2+|u_{x_i}(x) |^2 \, \right)^\frac{p_i -2}{2}| \tau_{j,h} u_{x_i} |^2 dx}_{\textbf{RHS}}  \notag \\
& \qquad + c_\varepsilon \int_{\Omega} \sum_{i=1}^{n}\, \left( \, |u_{x_i}(x+e_jh)|^2+|u_{x_i}(x) |^2 \, \right)^\frac{p_i -2}{2}  | \tau_{j,h} u |^2 \, |\eta_{x_i}|^2 dx  \notag \\
& \leq 2 \varepsilon \,\textbf{\textbf{RHS}} \,  + \frac{c_\varepsilon}{(t-s)^2}  \int_{B_t} \sum_{i=1}^{n}\, \left( \, |u_{x_i}(x+e_jh)|^2+|u_{x_i}(x) |^2 \, \right)^\frac{p_i -2}{2}  | \tau_{j,h} u |^2 dx  \notag \\
& \leq 2 \varepsilon \,\textbf{\textbf{RHS}} \,  + \frac{c_\varepsilon}{(t-s)^2}\sum_{i=1}^{n}  \Big(  \int_{B_t} \, \left( \, |u_{x_i}(x+e_jh)|^2+|u_{x_i}(x) |^2 \, \right)^\frac{p_i}{2}  dx \Big)^{\frac{p_i -2}{p_i}} \cdot \notag \\
& \qquad \qquad \qquad \ \ \ \ \ \cdot \Big( \int_{B_t} | \tau_{j,h} u|^{p_i} dx \Big)^{\frac{2}{p_i}}  \notag \\
& \leq 2 \varepsilon \,\textbf{\textbf{RHS}} \,  + \frac{c_\varepsilon}{(t-s)^2}\sum_{i=1}^{n} \Big(  \int_{B_{t'}} \,  \,|u_{x_i}(x) |^{p_i} \,  dx \Big)^{\frac{p_i -2}{p_i}}\cdot \Big( \int_{B_t} | \tau_{j,h} u|^{p_i} dx \Big)^{\frac{2}{p_i}},
\end{align}
where in the last two inequalities we used that $|D \eta| \leq \frac{c}{t-s}$, H\"older's inequality and Lemma \ref{ldiff}. For the estimate of the last term in \eqref{termtau}, we have to distinguish between two cases, as done in \cite{BraTau}.\\
\textbf{Case 1:} $\quad \mathbf{i \leq j \leq n}$. 
In this case we have that $p_i \leq p_j$ and this allows us to H\"older's inequality as follows
\begin{align}\label{tau2}
\int_{B_t} |\tau_{j,h} u|^{p_i} \, dx &\leq c(R, p_i, p_j) \left( \int_{B_t} |\tau_{j,h} u|^{p_j} \, dx \right)^{\frac{p_i}{p_j}} \leq c(R, p_i, p_j) |h|^{p_i} \left( \int_{B_{t'}} |u_{x_j}|^{p_j} \, dx \right)^{\frac{p_i}{p_j}}.
\end{align}
\textbf{Case 2:}  $\quad \mathbf{j+1 \leq i \leq n }$. Now, we have that $p_i \geq p_j$ but, by assumption \eqref{Ipotesip}, we have $p_i \le p_n \le p_1+2  \le p_j +2$ therefore by H\"older's inequality we have


\begin{align}\label{tau3}
\int_{B_t} |\tau_{j,h} u|^{p_i} \, dx &\leq c(R, p_i, p_j) \left( \int_{B_t} |\tau_{j,h} u|^{p_j+2} \right)^{\frac{p_i}{ p_j+2}} \leq c(R, p_i, p_j) |h|^{p_i} \left( \int_{B_{t'}} |u_{x_j}|^{p_j+2} \right)^{\frac{p_i}{ p_j+2}}. 
\end{align}
Let us explicitly remark that, by  assumption \eqref{Apriori} and Theorem \ref{ulimitatoCupini}, last integral in \eqref{tau3} is finite by the interpolation inequality of Proposition \ref{LemmaBrasco}.\\
Next, inserting \eqref{tau2} and \eqref{tau3} in \eqref{termtau}, we infer
\begin{align}
 |I_2|  & \leq 2 \varepsilon \,\textbf{\textbf{RHS}} \,  + \frac{c_\varepsilon}{(t-s)^2} \, |h|^2\sum_{i=1}^{j}  \, \Big(  \int_{B_{t'}} \,  \,|u_{x_i}(x) |^{p_i} \,  dx \Big)^{\frac{p_i -2}{p_i}} \cdot \left( \int_{B_{t'}} |u_{x_j}|^{p_j} \, dx  \right)^{\frac{2}{p_j}} \notag \\
& \qquad  + \frac{c_\varepsilon}{(t-s)^2} \, |h|^2\sum_{i=j+1}^{n}  \, \Big(  \int_{B_{t'}} \,  \,|u_{x_i}(x) |^{p_i} \,  dx \Big)^{\frac{p_i -2}{p_i}} \cdot \left( \int_{B_{t'}} |u_{x_j}|^{p_j+2} \right)^{\frac{p_i}{ p_j+2}} \notag \\
& \leq 2 \varepsilon \,\textbf{\textbf{RHS}} \,  +\frac{c_\varepsilon}{(t-s)^2} \, |h|^{2}  \sum_{i=1}^{j}  \, \Big(  \int_{B_{t'}} \,  \,|u_{x_i}(x) |^{p_i} \,  dx \Big)^{\frac{p_i -2}{p_i}} \cdot \left( \int_{B_{t'}} |u_{x_j}|^{p_j} \, dx  \right)^{\frac{2}{p_j}} \notag \\
& \qquad  + \frac{c_{\varepsilon,\beta}}{(t-s)^2} \, |h|^{2} \sum_{i=j+1}^{n}  \, \Big(  \int_{B_{t'}} \,  \,|u_{x_i}(x) |^{p_i} \,  dx \Big)^{\frac{(p_i -2)(p_j+2)}{p_i(p_j+2-p_i)}} + \beta |h|^{2} \left( \int_{B_{t'}} |u_{x_j}|^{p_j+2} \right), \notag 
\end{align}
where, in the last line, we used Young's inequality. Therefore, we conclude that
\begin{align}\label{I2}
    |I_2| & \leq  2 \varepsilon \,\textbf{\textbf{RHS}} \,  +\frac{c_\varepsilon}{(t-s)^2} \, |h|^{2}  \sum_{i=1}^{j}  \, \Big(  \int_{B_{t'}} \,  \,|u_{x_i}(x) |^{p_i} \,  dx \Big)^{\frac{p_i -2}{p_i}} \cdot \left( \int_{B_{t'}} |u_{x_j}|^{p_j} \, dx  \right)^{\frac{2}{p_j}} \notag \\
& \qquad  + \frac{c_{\varepsilon,\beta}}{(t-s)^2} \, |h|^{2} \sum_{i=j+1}^{n}  \, \Big(  \int_{B_{t'}} \,  \,|u_{x_i}(x) |^{p_i} \,  dx \Big)^{\frac{(p_i -2)(p_j+2)}{p_i(p_j+2-p_i)}} + \beta |h|^{2} \left( \int_{B_{t'}} |u_{x_j}|^{p_j+2} \right).
\end{align}
From condition \eqref{A4}, Lemma \ref{ldiff}, the properties of $\eta$ and H\"older's inequality, 
 we derive
\begin{align}\label{I3prima}
    |I_3| &= \left| \int_{\Omega} \langle f_{\xi} (x+e_jh, Du(x))- f_{\xi} (x, Du(x)) , 2 \eta D \eta \tau_{j,h} u \rangle \, dx  \right| \notag \\
    & \leq 2|h| \sum_{i=1}^n \int_{\Omega} \eta | \eta_{x_i}||\tau_{j, h} u| \, \Big(g(x+h)+ g(x) \Big) \,|u_{x_i}|^{p_i-1} dx \notag \\
    & \leq \frac{c}{t-s}|h| \sum_{i=1}^n \int_{B_t}  |\tau_{j, h} u| \, \Big(g(x+h)+ g(x) \Big) \,|u_{x_i}|^{p_i-1} dx \notag \\
      &  \leq  \frac{c}{t-s} \, |h| \, \left( \int_{B_{R}} g(x)^{r} dx \right)^{\frac{1}{r}}  \underbrace{\sum_{i=1}^n\left( \int_{B_{t}}|\tau_{j,h} u|^{\frac{r}{r-1}} |u_{x_i}|^{\frac{r(p_i-1)}{r-1}} dx \right)^{\frac{r-1}{r}}}_{J}.  
\end{align}
Notice that 
$$\frac{1}{p_j+2}\frac{r}{r-1} + \frac{p_i-1}{p_i+2}\frac{r}{r-1} < 1$$ 
if and only if
\begin{equation}\label{Stima2p}
  r> \frac{(p_j+2)(p_i+2)}{3(p_j+2)-(p_i+2)}.
\end{equation}
Inequality \eqref{Stima2p} is satisfied by virtue of assumption \eqref{Ipotesir}, once we observe that
$$ \frac{(p_j+2)(p_i+2)}{3(p_j+2)-(p_i+2)} \leq p_i+2 \Longleftrightarrow{ p_i \leq 2p_j +2 }$$
which is  satisfied by assumption \eqref{Ipotesip}. Since by hypothesis \eqref{Ipotesir}, we have that $$r> p_n + 2 > p_i+2, \quad \forall i=1, \dots, n,$$ then the inequality at \eqref{Stima2p} holds true.\\
Therefore, there exists $\theta \in (0,1)$ such that
$$\frac{\theta}{p_j+2}\frac{r}{r-1} + \frac{r}{r-1} \frac{1-\theta}{p_j} + \frac{p_i-1}{p_i+2}\frac{r}{r-1} = 1.$$
Thanks to Hölder's inequality and Lemma \ref{ldiff}, we get
\begin{align}\label{J}
    J &\leq  \left( \int_{B_{t}}|\tau_{j,h} u|^{p_j+2}  \, dx \right)^{\frac{\theta}{p_j+2}} \  \left( \int_{B_{t}}|\tau_{j,h} u|^{ p_j}  \, dx \right)^{\frac{1- \theta}{p_j}} \sum_{i=1}^{n}\left(   \int_{B_t} |u_{x_i}|^{p_i+2} \, dx \right)^{\frac{p_i-1}{p_i+2}} \notag \\
   & \leq  c|h| \ \left( \int_{B_{{t'}}}|u_{x_j}|^{p_j+2}  \, dx \right)^{\frac{\theta}{p_j+2}} \  \left( \int_{B_{{t'}}}|u_{x_j}|^{ p_j}  \, dx \right)^{\frac{1- \theta}{p_j}} \sum_{i=1}^{n} \left(\int_{B_t} |u_{x_i}|^{p_i+2} \, dx \right)^{\frac{p_i-1}{p_i+2}}\notag \\
    & \leq c |h| \ \left( \int_{B_{{t'}}}| u_{x_j}|^{ p_j}  \, dx \right)^{\frac{1- \theta}{p_j}} \sum_{i=1}^{n}\left(  \int_{B_{t'}} |u_{x_i}|^{p_i+2} \, dx \right)^{  \frac{\theta+ p_i-1}{p_i+2}}. 
\end{align} 
Now inserting \eqref{J} in \eqref{I3prima}, using Young's inequality and denoting by $\alpha$ the conjugate of exponent $\frac{p_i+2}{\theta+(p_i -1)}$,  we have
\begin{align}\label{I3}
    |I_3| &\leq   \frac{c_{ \beta}}{t-s} \, |h|^{2} \sum_{i=1}^{n} \left( \int_{B_{R}} g(x)^{r} \, dx \right)^{\frac{\alpha }{r}}  \left( \int_{B_{t'}} |u_{x_i}|^{p_i} \, dx \right)^{\frac{\alpha(1- \theta)}{p_i}} \notag \\
    & \qquad + |h|^{2} \beta \sum_{i=1}^{n} \left( \int_{B_{t'}} |u_{x_i}|^{p_i+2} \, dx \right)
\end{align}
Similarly as for the estimate of $I_3$, we obtain
\begin{align}\label{IntI4prima}
     |I_4| &= \left|  \int_{\Omega} \langle f_{\xi} (x+e_jh, Du(x))- f_{\xi} (x, Du(x)) , \eta^2  \tau_{j,h} Du \rangle \, dx \right| \notag \\
    & \leq c |h| \sum_{i=1}^{n}  \int_{\Omega} \eta^2 |\tau_{j, h}u_{x_i}| \, (g(x+h)+ g(x))\, |u_{x_i}|^{p_i -1} dx \notag \\
    & \leq \varepsilon \ \ \ \ \underbrace{\int_{\Omega} \eta^2 \sum_{i=1}^{n}\, \left( \, |u_{x_i}(x+e_jh)|^2+|u_{x_i}(x) |^2 \, \right)^\frac{p_i -2}{2}| \tau_{j,h} u_{x_i} |^2 dx}_{\textbf{\textbf{RHS}}}  \notag \\
    & \qquad + c_{\varepsilon} |h|^2 \sum_{i=1}^{n}  \int_{\Omega} \eta^2 |u_{x_i}|^{p_i} \, (g(x+h)+ g(x))^2\,  dx \notag \\
     & \leq \varepsilon \textbf{\textbf{RHS}} + c_{\varepsilon} |h|^{2} \sum_{i=1}^{n}  \left( \int_{B_{R}} g(x)^{r} dx \right)^{\frac{2}{r}} \, \left( \int_{B_{t}} |u_{x_i}|^{\frac{p_i r}{r-2}} dx \right)^{\frac{r-2}{r}}.
\end{align}
We note that
$$\frac{(p_i+2)}{p_ir}(r-2)>1 \Longleftrightarrow{r > p_i+2}$$
which is true by assumption \eqref{Ipotesir}.
Therefore applying H\"older's inequality in \eqref{IntI4prima}, we infer 
\begin{align}\label{I4}
 |I_4| &\leq  \varepsilon \textbf{\textbf{RHS}} +    c_{\varepsilon} |h|^{2} \sum_{i=1}^{n}  \left( \int_{B_{R}} g(x)^{r} dx \right)^{\frac{2}{r}} \,   \left( \int_{B_{t}} |u_{x_i}|^{p_i+2} dx \right)^{\frac{p_i}{p_i +2}} \notag \\
 & \leq \varepsilon \textbf{\textbf{RHS}} +  c_{\varepsilon, \beta} |h|^2 \sum_{i=1}^{n} \left( \int_{B_{R}} g(x)^{r} \, dx \right)^{\frac{p_i+2}{r}} +  |h|^2 \beta \sum_{i=1}^{n}  \left( \int_{B_{{t'}}} |u_{x_i}|^{p_i+2} dx \right) \notag \\
 & \leq \varepsilon \textbf{\textbf{RHS}} +  c_{\varepsilon, \beta} |h|^2 \left( \int_{B_{R}} g(x)^{r} \, dx \right)^{\gamma} +  |h|^2 \beta \sum_{i=1}^{n}  \left( \int_{B_{{t'}}} |u_{x_i}|^{p_i+2} dx \right),
\end{align}
where we used Young's inequality with exponents $ \frac{p_i +2}{p_i}$ and $\frac{p_i+2}{2}$ and $\gamma=\gamma(p_i)$.\\
Inserting \eqref{I1}, \eqref{I2}, \eqref{I3} and \eqref{I4} in \eqref{SommaInt}, we get
\begin{align}\label{Formulacompleta}
 l &\int_{\Omega} \eta^2 \sum_{i=1}^{n}\, \left( \, |u_{x_i}(x+e_jh)|^2+|u_{x_i}(x) |^2 \, \right)^\frac{p_i -2}{2}| \tau_{j,h} u_{x_i} |^2 dx \notag \\
& \qquad \leq  3 \varepsilon \, \int_{\Omega}  \eta^2 \sum_{i=1}^{n}\, \left( \, |u_{x_i}(x+e_jh)|^2+|u_{x_i}(x) |^2 \, \right)^\frac{p_i -2}{2}| \tau_{j,h} u_{x_i} |^2 dx \notag \\ 
& \qquad +\frac{c_\varepsilon}{(t-s)^2} \, |h|^{2}  \sum_{i=1}^{j}  \, \Big(  \int_{B_{t'}} \,  \,|u_{x_i}(x) |^{p_i} \,  dx \Big)^{\frac{p_i -2}{p_i}} \cdot \left( \int_{B_{t'}} |u_{x_j}|^{p_j} \, dx  \right)^{\frac{2}{p_j}} \notag \\
& \qquad  + \frac{c_{\varepsilon,\beta}}{(t-s)^2} \, |h|^{2} \sum_{i=j+1}^{n}  \, \Big(  \int_{B_{t'}} \,  \,|u_{x_i}(x) |^{p_i} \,  dx \Big)^{\frac{(p_i -2)(p_j+2)}{p_i(p_j+2-p_i)}} + 3 \beta |h|^2 \sum_{i=1}^{n} \left( \int_{B_{t'}} |u_{x_i}|^{p_i+2} \right)\notag \\
& \qquad + \frac{c_{\beta}}{t-s} \, |h|^{2} \sum_{i=1}^{n} \left( \int_{B_{R}} g(x)^{r} \, dx \right)^{\frac{\alpha }{r}}  \left( \int_{B_{{t'}}} |u_{x_i}|^{p_i} \, dx \right)^{\frac{\alpha(1- \theta)}{p_i}} \notag \\
     & \qquad +   c_{\varepsilon, \beta} |h|^2  \left( \int_{B_{R}} g(x)^{r} \, dx \right)^{^{\gamma}}.  
\end{align}
Choosing $\varepsilon = \frac{l}{6}$, we can reabsorb the first integral in the right-hand side of \eqref{Formulacompleta} by the left-hand side thus getting
\begin{align}
 &\int_{\Omega} \eta^2 \sum_{i=1}^{n}\, \left( \, |u_{x_i}(x+e_jh)|^2+|u_{x_i}(x) |^2 \, \right)^\frac{p_i -2}{2}| \tau_{j,h} u_{x_i} |^2 dx \notag \\
&  \qquad \leq \frac{c}{(t-s)^2} \, |h|^{2}  \sum_{i=1}^{j}  \, \Big(  \int_{B_{t'}} \,  \,|u_{x_i}(x) |^{p_i} \,  dx \Big)^{\frac{p_i -2}{p_i}} \cdot \left( \int_{B_{t'}} |u_{x_j}|^{p_j} \, dx  \right)^{\frac{2}{p_j}} \notag \\
& \qquad  + \frac{c_{\beta}}{(t-s)^2} \, |h|^{2} \sum_{i=j+1}^{n}  \, \Big(  \int_{B_{t'}} \,  \,|u_{x_i}(x) |^{p_i} \,  dx \Big)^{\frac{(p_i -2)(p_j+2)}{p_i(p_j+2-p_i)}} + 3 \beta |h|^2 \sum_{i=1}^{n} \left( \int_{B_{t'}} |u_{x_i}|^{p_i+2} \right)\notag \\
& \qquad + \frac{c_{\beta}}{t-s} \, |h|^{2} \sum_{i=1}^{n} \left( \int_{B_{R}} g(x)^{r} \, dx \right)^{\frac{\alpha }{r}}  \left( \int_{B_{{t'}}} |u_{x_i}|^{p_i} \, dx \right)^{\frac{\alpha(1- \theta)}{p_i}} \notag \\
     & \qquad +   c_{ \beta} |h|^2 \left( \int_{B_{R}} g(x)^{r} \, dx \right)^{^{\gamma}} .
\end{align}
Then, by Lemmas \ref{D1} and \ref{Lemmahzero}, letting $|h| \to 0,$ we obtain
\begin{align}\label{sommauxij}
 &    \sum_{i=1}^{n}   \left( \int_{B_s} |u_{x_i x_j}|^2|u_{x_i}|^{p_i-2} dx \right) \notag \\
&  \qquad \leq \frac{c}{(t-s)^2} \,   \sum_{i=1}^{j}  \, \Big(  \int_{B_{R}} \,  \,|u_{x_i}(x) |^{p_i} \,  dx \Big)^{\frac{p_i -2}{p_i}} \cdot \left( \int_{B_{R}} |u_{x_j}|^{p_j} \, dx  \right)^{\frac{2}{p_j}} \notag \\
& \qquad  + \frac{c}{(t-s)^2} \,  \sum_{i=j+1}^{n}  \, \Big(  \int_{B_{R}} \,  \,|u_{x_i}(x) |^{p_i} \,  dx \Big)^{\frac{(p_i -2)(p_j+2)}{p_i(p_j+2-p_i)}}  \notag \\
& \qquad + \frac{c_{\beta}}{t-s} \,  \sum_{i=1}^{n} \left( \int_{B_{R}} g(x)^{r} \, dx \right)^{\frac{\alpha }{r}}  \left( \int_{B_R} |u_{x_i}|^{p_i} \, dx \right)^{\frac{\alpha(1- \theta)}{p_i}} \notag \\
     & \qquad + c_{\beta}  \left( \int_{B_{R}} g(x)^{r} \, dx \right)^{^{\gamma}}
     + 3\beta \sum_{i=1}^{n} \left( \int_{B_{t'}} |u_{x_i}|^{p_i+2} \right),
\end{align}
which  in particular yields for every $i=1, \dots, n$
\begin{align}
 &       \int_{B_s} |u_{x_i x_i}|^2|u_{x_i}|^{p_i-2} dx   \leq \frac{c}{(t-s)^2} \,   \sum_{h=1}^{i}  \, \Big(  \int_{B_{R}} \,  \,|u_{x_h}(x) |^{p_h} \,  dx \Big)^{\frac{p_h -2}{p_h}} \cdot \left( \int_{B_{R}} |u_{x_i}|^{p_i} \, dx  \right)^{\frac{2}{p_i}} \notag \\
& \qquad  + \frac{c}{(t-s)^2} \,  \sum_{h=i+1}^{n}  \, \Big(  \int_{B_{R}} \,  \,|u_{x_h}(x) |^{p_h} \,  dx \Big)^{\frac{(p_h -2)(p_i+2)}{p_h(p_i+2-p_h)}}  \notag \\
& \qquad + \frac{c_{\beta}}{t-s} \,  \sum_{h=1}^{n} \left( \int_{B_{R}} g(x)^{r} \, dx \right)^{\frac{\alpha }{r}}  \left( \int_{B_R} |u_{x_h}|^{p_h} \, dx \right)^{\frac{\alpha(1- \theta)}{p_h}} \notag \\
     & \qquad + c_{\beta}   \left( \int_{B_{R}} g(x)^{r} \, dx \right)^{^{\gamma}}
      + 3\beta \sum_{h=1}^{n} \left( \int_{B_{t'}} |u_{x_h}|^{p_h+2} \right).
\end{align}
Summing over $i=1,\dots, n$, the previous inequality, we infer
\begin{align}\label{sommasuj}
     \sum_{i=1}^n & \left(   \int_{B_s} |u_{x_i x_i}|^2|u_{x_i}|^{p_i-2} dx  \right) \notag \\
 &\leq  \frac{c}{(t-s)^2} \,  \sum_{i=1}^{n}\sum_{h=1}^{i} \left(  \int_{B_{R}} \,  \,|u_{x_h}(x) |^{p_h} \,  dx \right)^{\frac{p_h -2}{p_h}} \cdot \left( \int_{B_{R}} |u_{x_i}(x)|^{p_i} \, dx  \right)^{\frac{2}{p_i}} \notag \\
& \qquad  + \frac{c}{(t-s)^2} \,   \sum_{i=1}^{n} \sum_{h=i+1}^{n}  \, \Big(  \int_{B_{R}} \,  \,|u_{x_h}(x) |^{p_h} \,  dx \Big)^{\frac{(p_h -2)(p_i+2)}{p_h(p_i+2-p_h)}}   \notag \\
& \qquad +n\Bigg[  \frac{c_{\beta}}{t-s} \,  \sum_{h=1}^{n} \left( \int_{B_{R}} g(x)^{r} \, dx \right)^{\frac{\alpha }{r}}  \left( \int_{B_R} |u_{x_h}|^{p_h} \, dx \right)^{\frac{\alpha(1- \theta)}{p_h}} \notag \\
     & \qquad + c_{\beta}   \left( \int_{B_{R}} g(x)^{r} \, dx \right)^{^{\gamma}}
      + 3\beta \sum_{h=1}^{n} \left( \int_{B_{t'}} |u_{x_h}|^{p_h+2} \right)\Bigg] \notag \\
      &  \leq \frac{c}{(t-s)^2} \,   \sum_{i=1}^{n} \left( 1+ \sum_{h=1}^{n}\int_{B_{R}} \,  \,|u_{x_h}(x) |^{p_h} \,  dx \right)^{\frac{p_i -2}{p_i}} \cdot  \sum_{i=1}^{n} \left( 1 + \sum_{h=1}^{n}\int_{B_{R}} |u_{x_h}(x)|^{p_h} \, dx  \right)^{\frac{2}{p_i}} \notag \\
      & \qquad  + \frac{c}{(t-s)^2} \,   \sum_{i=1}^{n} \sum_{h=i+1}^{n}  \, \Big(  \int_{B_{R}} \,  \,|u_{x_h}(x) |^{p_h} \,  dx \Big)^{\frac{(p_h -2)(p_i+2)}{p_h(p_i+2-p_h)}}   \notag \\
& \qquad +n\Bigg[  \frac{c_{\beta}}{t-s} \,  \sum_{h=1}^{n} \left( \int_{B_{R}} g(x)^{r} \, dx \right)^{\frac{\alpha }{r}}  \left( \int_{B_R} |u_{x_h}|^{p_h} \, dx \right)^{\frac{\alpha(1- \theta)}{p_h}} \notag \\
     & \qquad + c_{\beta}   \left( \int_{B_{R}} g(x)^{r} \, dx \right)^{^{\gamma}}
      + 3\beta \sum_{h=1}^{n} \left( \int_{B_{t'}} |u_{x_h}|^{p_h+2} \right)\Bigg] \notag \\
        &  \leq \frac{c}{(t-s)^2} \,   \sum_{i=1}^{n} \left( 1+ \sum_{h=1}^{n}\int_{B_{R}} \,  \,|u_{x_h}(x) |^{p_h} \,  dx \right)^{\tilde{\alpha}} \cdot  \sum_{i=1}^{n} \left( 1 + \sum_{h=1}^{n}\int_{B_{R}} |u_{x_h}(x)|^{p_h} \, dx  \right)^{\tilde{\beta}} \notag \\
      & \qquad  + \frac{c}{(t-s)^2} \,   \sum_{i=1}^{n} \sum_{h=i+1}^{n}  \, \Big(  \int_{B_{R}} \,  \,|u_{x_h}(x) |^{p_h} \,  dx \Big)^{\frac{(p_h -2)(p_i+2)}{p_h(p_i+2-p_h)}}   \notag \\
& \qquad +n\Bigg[  \frac{c_{\beta}}{t-s} \,  \sum_{h=1}^{n} \left( \int_{B_{R}} g(x)^{r} \, dx \right)^{\frac{\alpha }{r}}  \left( \int_{B_R} |u_{x_h}|^{p_h} \, dx \right)^{\frac{\alpha(1- \theta)}{p_h}} \notag \\
     & \qquad + c_{\beta}   \left( \int_{B_{R}} g(x)^{r} \, dx \right)^{^{\gamma}}
      + 3\beta \sum_{h=1}^{n} \left( \int_{B_{t'}} |u_{x_h}|^{p_h+2} \right)\Bigg] \notag \\
       &  \leq \frac{c}{(t-s)^2} \,  n \left( 1+ \sum_{h=1}^{n}\int_{B_{R}} \,  \,|u_{x_h}(x) |^{p_h} \,  dx \right)^{\tilde{\alpha}} \cdot n \left( 1 + \sum_{h=1}^{n}\int_{B_{R}} |u_{x_h}(x)|^{p_h} \, dx  \right)^{\tilde{\beta}} \notag \\
      & \qquad  + \frac{c}{(t-s)^2} \,   \sum_{i=1}^{n} \sum_{h=i+1}^{n}  \, \Big(  \int_{B_{R}} \,  \,|u_{x_h}(x) |^{p_h} \,  dx \Big)^{\frac{(p_h -2)(p_i+2)}{p_h(p_i+2-p_h)}}   \notag \\
& \qquad +n\Bigg[  \frac{c_{\beta}}{t-s} \,  \sum_{h=1}^{n} \left( \int_{B_{R}} g(x)^{r} \, dx \right)^{\frac{\alpha }{r}}  \left( \int_{B_R} |u_{x_h}|^{p_h} \, dx \right)^{\frac{\alpha(1- \theta)}{p_h}} \notag \\
     & \qquad + c_{\beta}   \left( \int_{B_{R}} g(x)^{r} \, dx \right)^{^{\gamma}}
      + 3\beta \sum_{h=1}^{n} \left( \int_{B_{t'}} |u_{x_h}|^{p_h+2} \right)\Bigg] \notag \\
       &  \leq \frac{n^2 c}{(t-s)^2} \,   \left( 1+ \sum_{h=1}^{n}\int_{B_{R}} \,  \,|u_{x_h}(x) |^{p_h} \,  dx \right)^{\tilde{\alpha}+ \tilde{\beta}}  \notag \\
      & \qquad  + \frac{c}{(t-s)^2} \,   \sum_{i=1}^{n} \sum_{h=i+1}^{n}  \, \Big(  \int_{B_{R}} \,  \,|u_{x_h}(x) |^{p_h} \,  dx \Big)^{\frac{(p_h -2)(p_i+2)}{p_h(p_i+2-p_h)}}   \notag \\
& \qquad +n\Bigg[  \frac{c_{\beta}}{t-s} \,  \sum_{h=1}^{n} \left( \int_{B_{R}} g(x)^{r} \, dx \right)^{\frac{\alpha }{r}}  \left( \int_{B_R} |u_{x_h}|^{p_h} \, dx \right)^{\frac{\alpha(1- \theta)}{p_h}} \notag \\
     & \qquad + c_{\beta}   \left( \int_{B_{R}} g(x)^{r} \, dx \right)^{^{\gamma}}
      + 3\beta \sum_{h=1}^{n} \left( \int_{B_{t'}} |u_{x_h}|^{p_h+2} \right)\Bigg],
\end{align}
where $\tilde{\alpha}= \max \{ \frac{p_i-2}{p_i}, i=1, \dots, n\}$ and $\tilde{\beta}= \max \{ \frac{2}{p_i}, i=1, \dots, n\}$.\\ 
From the hypothesis
 $$V_{p_i}(u_{x_i}) \in W^{1,2}_{\mathrm{loc}}(\Omega,\mathbb{R})$$
and Proposition \ref{LemmaBrasco} together Theorem with \ref{ulimitatoCupini}, we have
$$u_{x_i} \in L^{p_i+2}_{\mathrm{loc}}(\Omega ,\mathbb{R}), \quad \forall i= 1, \dots, n $$
and the following estimate holds for every $i=1,\dots, n$
 \begin{align}\label{TesiProprof}
      \int_{B_{{t'}}} |u_{x_i}|^{p_i+2} dx &\leq c \Vert u \Vert_{\infty} \left[   \left( \int_{B_r} |u_{x_i x_i}|^2|u_{x_i}|^{p_i-2} dx \right)+ \frac{1}{(r-t')^\gamma}  \left(\int_{B_r} |u_{x_i}|^{p_i} dx \right)^{\gamma} \right] \notag \\
      & \leq c \Vert u \Vert_{\infty} \left[   \left( \int_{B_r} |u_{x_i x_i}|^2|u_{x_i}|^{p_i-2} dx \right)+ \frac{1}{(r-t')^\gamma}  \left(\int_{B_R} |u_{x_i}|^{p_i} dx \right)^{\gamma} \right],
   \end{align}
   with $c=c(n, p_i)$ and $\gamma=\gamma(p_i).$\\
Summing the inequality \eqref{TesiProprof} over  $i=1,\dots, n$ and inserting the resulting inequality in \eqref{sommasuj}, we find that
\begin{align}\label{uxipi+2}
 &  \sum_{i=1}^n \left(   \int_{B_s} |u_{x_i x_i}|^2|u_{x_i}|^{p_i-2} dx   \right)  \notag \\ 
& \leq  \frac{n^2 c}{(t-s)^2} \,   \left( 1+ \sum_{i=1}^{n}\int_{B_{R}} \,  \,|u_{x_i}(x) |^{p_i} \,  dx \right)^{\tilde{\alpha}+ \tilde{\beta}}   \notag \\
& \qquad  + \frac{c_2}{(t-s)^2} \,  \sum_{j=1}^{n} \sum_{i=j+1}^{n}  \, \Big(  \int_{B_{R}} \,  \,|u_{x_i}(x) |^{p_i} \,  dx \Big)^{\frac{(p_i -2)(p_j+2)}{p_i(p_j+2-p_i)}}  \notag \\
& \qquad + \frac{c_1}{t-s} \,  \sum_{i=1}^{n} \left( \int_{B_{R}} g(x)^{r} \, dx \right)^{\frac{\alpha }{r}}  \left( \int_{B_R} |u_{x_i}|^{p_i} \, dx \right)^{\frac{\alpha(1- \theta)}{p_i}} \notag \\
     & \qquad + c_{3}  \left( \int_{B_{R}} g(x)^{r} \, dx \right)^{\frac{p_n+2}{r}}  \notag \\
    &  \qquad + \beta \,\tilde{c}\lVert u \rVert_{\infty} \sum_{i=1}^{n}\left[   \left( \int_{B_r} |u_{x_i x_i}|^2|u_{x_i}|^{p_i-2} dx \right) + \frac{c_4}{(r-t')^\gamma}  \left(\int_{B_R} |u_{x_i}|^{p_i} dx \right)^{\gamma} \right] .
\end{align}
Setting  $$ \phi(\rho)= \sum_{i=1}^n   \left( \int_{B_\rho} |u_{x_i x_i}|^2|u_{x_i}|^{p_i-2} dx \right) dx, $$
inequality \eqref{uxipi+2} can be written as follows
    \begin{equation}\label{phirho}
\phi(\rho) \leq \beta \tilde{c} \lVert u \rVert_{\infty} \phi(r)+ c_3 + \frac{c_1}{t-s} + \frac{c_2}{(t-s)^2} + \frac{c_4}{(r-t')^\gamma}
\end{equation}
Choosing $\beta >0$ such that $ \theta =\beta \tilde{c} \lVert u \rVert_{\infty}<1$ and since \eqref{phirho} holds true for every radii $\rho<s<t<t'<r$, we may select radii selecting radii
  $$s= \rho + \frac{r-\rho}{4}, \,t = \rho + \frac{r-\rho}{2}, \, t'= \rho + \frac{3(r-\rho)}{4},$$
  so that $t-s=r-t'=\frac{r- \rho}{4}.$ Hence,  \eqref{phirho} becomes
  \begin{equation}\label{phir}
\phi(\rho) \leq \theta  \phi(r)+ c_3 + \frac{\tilde{c}_1}{r-\rho} + \frac{\tilde{c}_2}{(r-\rho)^2} + \frac{\tilde{c}_4}{(r-\rho)^\gamma}.
\end{equation}
Since \eqref{phir} holds true for every $\frac{R}{4}<\rho<r<R$, we are legitimate to apply Lemma \ref{lm3} to the function $\phi: [\frac{R}{4},R] \to \mathbb{R}$, thus getting
\begin{equation}\label{phiR}
    \phi \left(\frac{R}{4} \right) \leq  c_3 + \frac{\tilde{c}_1}{R} + \frac{\tilde{c}_2}{R^2} + \frac{\tilde{c}_4}{R^\gamma}.
    \end{equation}
   Recalling the definition of $\phi(t)$, \eqref{phiR} yields that 
    \begin{align}\label{Stimacompattauxipi+2}
        \sum_{i=1}^{n}   \left( \int_{B_{\frac{R}{4}}} |u_{x_i x_i}|^2|u_{x_i}|^{p_i-2} dx \right) & \leq c  \left( \sum_{i=1}^{n} \Vert u_{x_i} \Vert_{L^{p_i}(B_R)} + \Vert g \Vert_{L^r (B_{R})} \right)^\sigma ,
\end{align}
where $c=c(p_i,n,l L, R)$ and $\sigma=\sigma(n,p_i)$
Putting \eqref{Stimacompattauxipi+2} in \eqref{TesiProprof}, we derive also
\begin{align}\label{pi+2}
 \sum_{i=1}^n \int_{B_{\frac{R}{4}}} |u_{x_i}|^{p_i+2} dx & \leq c \left( \sum_{i=1}^{n} \Vert u_{x_i} \Vert_{L^{p_i}(B_R)} + \Vert g \Vert_{L^r (B_{R})} \right)^\sigma
     , 
\end{align}
for positive constants $c = c(n,p_i,\lambda, \Lambda,R)$ and $\sigma= \sigma (n,p_i)$.\\
Now, inserting \eqref{pi+2} in \eqref{sommauxij}, we infer for every $i=1, \dots, n$ and  for every $j=1, \dots, n.$
\begin{align}\label{Stimacompattauxixjpi+2}
        \sum_{i=1}^{n}   \left( \int_{B_{\frac{R}{4}}} |u_{x_i x_j}|^2|u_{x_i}|^{p_i-2} dx \right) & \leq c  \left( \sum_{i=1}^{n} \Vert u_{x_i} \Vert_{L^{p_i}(B_R)} + \Vert g \Vert_{L^r (B_{R})} \right)^\sigma. 
\end{align}
which concludes the proof.
\end{proof}
\section{Approximation}\label{Appro}
In order to perform the approximation argument, let us fix a non-negative smooth kernel $\phi \in \mathcal{C}^\infty_0(B_1(0))$ such that $\int_{B_1(0)} \phi =1$ and consider the corresponding family of mollifiers $(\phi_k)_{k >0}$. We set, for a given ball $B_R \Subset \Omega$ ,
$$a_i^k=a_i * \phi_k,$$
$$g_k(x)=\underset{i}{\max} \{ |D a_i^k (x)|\}$$
and, for $x \in B_R$, we define
\begin{equation}\label{fvarepsilon}
f^k (x,\xi)= \sum_{i=1}^{n} a_i^k(x)|\xi|^{p_i} ,
\end{equation}
 for every $k < \text{dist}(B_R,\Omega)$. We shall use the following $$\mathcal{F}^{k}(u,B_R):=\int_{_{B_R} }f^{k}(x, \xi) dx$$ and
$$\mathcal{F}^{k}_{\varepsilon}(u,B_R):= \, \int_{B_R} \left( f^k(x,\xi)+ \varepsilon(1+ |\xi|^2)^{\frac{p_n}{2}} \right) dx.$$
\begin{flushleft}
One can easily check that $f^k(x, \xi)$ satisfies assumptions \eqref{A2}--\eqref{A3} and \eqref{A'4}, with $K = \underset{i}{\max} \left\{ \sup_{B_R} |a_i^k(x)| \right\}$.
\end{flushleft}
We consider the variational problems
\begin{equation}\label{Pfj}
    \inf \left\{  \int_{B_R} \left( f^k(x,Dv)+ \varepsilon(1+ |Dv|^2)^{\frac{p_n}{2}} \right) dx \ : \ v \in W^{1,p_n}_0(B_R, \mathbb{R}^N)+u^\eta \right\},
\end{equation}
where $$u^\eta = u * \varsigma_\eta$$ is the mollification of the local minimizer $u$ of \eqref{functional}, for a sequence of mollifier $\varsigma_n$.\\
It is well known that, by the direct methods of the calculus of variations, there exists a unique solution $u^\eta_{k,\varepsilon} \in W^{1,p_n}_0(B_R, \mathbb{R}^N)+ u^\eta $ of the problem \eqref{Pfj}. Since the integrand $f^k(x,Dv)+ \varepsilon(1+ |Dv|^2)^{\frac{p_n}{2}}$ satisfies the assumptions of  Theorem \ref{thm2Mascolo}, we have that $$V_{p_i}((u^\eta_{k,\varepsilon})_{x_i}) \in W_{\mathrm{loc}}^{1,2}(B_R).$$ Hence we are legitimate to use estimate \eqref{uxistima} of Theorem \ref{AppThm}, to obtain that
\begin{align*}
    \sum_{i=1}^{n} \left(  \int_{B_{R/4}}|(u^\eta_{k,\varepsilon})_{x_i}|^{p_i +2 } dx \right) 
    & \leq c \left( \sum_{i=1}^{n} \int_{B_R}\, \lvert(u^\eta_{k,\varepsilon})_{x_i} \rvert^{p_i} \, dx + \Vert g_k \Vert_{L^r (B_{R})} \right)^\sigma \\
     &  \leq c(\lambda) \left( \sum_{i=1}^{n} \int_{B_R}a_i^k(x)\, \lvert (u^\eta_{k,\varepsilon})_{x_i}\rvert^{p_i} \, dx + \Vert g_k \Vert_{L^r (B_{R})} \right)^\sigma \\
    & = c  \left( \sum_{i=1}^{n} \int_{B_R} f^k(x,D u^\eta_{k,\varepsilon}) dx+ \Vert g_k \Vert_{L^r (B_{R})} \right)^\sigma \\
    & \leq c  \left( \sum_{i=1}^{n} \int_{B_R} \left(f^k(x,D u^\eta_{k,\varepsilon}) + \varepsilon(1+|Du^\eta_{k,\varepsilon}|^2)^{\frac{p_n}{2}} \right) dx+ \Vert g_k \Vert_{L^r (B_{R})} \right)^\sigma  ,
\end{align*}
where  $c$ is a positive constant that is independent on $\varepsilon, k$ and $\eta$ and where  we used that $a_i^k(x)\geq \lambda, \forall i,k$. Using  the minimality of $u^\eta_{k,\varepsilon}$ in the right-hand side of previous estimate, we have that
\begin{align}\label{stimaukeni}
    \sum_{i=1}^{n} \left(  \int_{B_{R/4}}|(u^\eta_{k,\varepsilon})_{x_i}|^{p_i +2 } dx \right) & \leq c  \left( \sum_{i=1}^{n} \int_{B_R} \left(f^k(x,D u^\eta_{k,\varepsilon}) + \varepsilon(1+|Du^\eta_{k,\varepsilon}|^2)^{\frac{p_n}{2}} \right) dx+ \Vert g_k \Vert_{L^r (B_{R})} \right)^\sigma \notag \\
     & \leq c  \left( \sum_{i=1}^{n} \int_{B_R} \left(f^k(x,D u^\eta) + (1+|Du^\eta|^2)^{\frac{p_n}{2}} \right) dx+ \Vert g_k \Vert_{L^r (B_{R})} \right)^\sigma,
\end{align}
where $c$ is a positive constant that does not depend on $\eta, k$ and $\varepsilon$.\\
    Since $u^\eta \in W^{1,p_n}(B_R, \mathbb{R}^N)$, the right-hand side of the previous estimate is finite, and therefore, the left-hand side of the previous estimate is uniformly bounded  with respect to $\varepsilon$ for $\varepsilon \in (0,1)$. Therefore there exists $v^\eta_k$ such that, up to a subsequence,
    $$u^\eta_{k,\varepsilon} \rightharpoonup v^\eta_k \text{  weakly in  } W^{1, \textbf{p}+2}(B_R, \mathbb{R}^N) \, \text{ as  } \varepsilon \to 0.$$
By the weak lower semicontinuity of the norm, we get
\begin{align}\label{stimaconv1}
    \sum_{i=1}^{n} \left(  \int_{B_{R/4}}|(v^\eta_k)_{x_i}|^{p_i +2 } dx \right) & \leq \liminf_{\varepsilon}\sum_{i=1}^{n} \left(  \int_{B_{R/4}}|(u^\eta_{k,\varepsilon})_{x_i}|^{p_i +2 } dx \right) \notag \\
    & \leq c \liminf_{\varepsilon}    \left( \sum_{i=1}^{n}  \int_{B_R} \left(f^k(x,D u^\eta) + \varepsilon(1+|Du^\eta|^2)^{\frac{p_n}{2}} \right) dx+ \Vert g_k \Vert_{L^r (B_{R})} \right)^\sigma  \notag \\
    &= c  \left( \sum_{i=1}^{n} \int_{B_R}f^k(x,D u^\eta)  dx+ \Vert g_k \Vert_{L^r (B_{R})} \right)^\sigma ,
\end{align}
where we used \eqref{stimaukeni} and that, since $u^\eta \in W^{1, p_n}(\Omega)$, it holds 
\begin{equation}\label{fkappalim}
\liminf_{\varepsilon \to 0} \, \varepsilon \int_{B_R}(1+|Du^\eta|^2)^{\frac{p_n}{2}} dx =0. 
\end{equation}
 Using the definition of $f_k$ at \eqref{fvarepsilon}, we have that
\begin{equation}\label{convFk}
    \left|  \int_{B_R}   f^k(x,Du^\eta) -  f(x,Du^\eta) dx \right|\leq \sum_{i=1}^n \int_{B_R} |a^k_i(x)- a_i(x)|\, | u^\eta_{x_i}|^{p_i} dx .
\end{equation}
 So the fact that 
 $$a^k_i(x) \to a_i(x)\quad \text{uniformly on compact sets }$$ gives that
$$\lim_{k \to 0^+} \sum_{i=1}^n \lvert \lvert a^k_i(x)- a_i(x) \rvert \rvert_{\infty} \int_{B_R} \, | u^{\eta}_{x_i}|^{p_i} dx =0. $$
Therefore, passing to the limit as $k \to 0^+$ in \eqref{convFk}, yields
\begin{equation}\label{fktoF}
   \int_{B_R} f^k(x, Du^\eta) dx \to \int_{B_R}f(x,Du^\eta) dx.
    \end{equation}
By the properties of mollification it holds $g_k \to g$ in $L^r(B_R)$, and so, by \eqref{fktoF}, the right-hand side of \eqref{stimaconv1} is uniformly bounded w.r.t. to $k$. Therefore there exists $v^\eta$ such that, up to a subsequence,
    $$v^\eta_{k} \rightharpoonup v^\eta \text{  weakly in  } W^{1, \textbf{p}+2}(B_R, \mathbb{R}^N) \, \text{ as  } k \to 0.$$
Again by the weak lower semicontinuity of the norm, \eqref{stimaconv1} and \eqref{fktoF}, we get
\begin{align}\label{stimaconv2}
    \sum_{i=1}^{n} \left(  \int_{B_{R/4}}|(v^\eta)_{x_i}|^{p_i +2 } dx \right) & \leq \liminf_{k}\sum_{i=1}^{n} \left(  \int_{B_{R/4}}|(v^\eta_{k})_{x_i}|^{p_i +2 } dx \right) \notag \\
    &\leq c  \left( \sum_{i=1}^{n} \int_{B_R}f(x,D u^\eta)  dx+ \Vert g \Vert_{L^r (B_{R})} \right)^\sigma .
\end{align}
Now, we note that
\begin{align}\label{convF}
    \left|  \int_{B_R}   f(x,Du^\eta) -  f(x,Du) dx \right| &\leq \left| \sum_{i=1}^n \int_{B_R} \Big( a_i(x)\,  |u^\eta_{x_i}|^{p_i}-a_i(x)  |u_{x_i}|^{p_i} \Big)dx \right| \notag \\
     &\leq  \sum_{i=1}^n \int_{B_R}a_i(x)\,\Big|  |u^\eta_{x_i}|^{p_i} -|u_{x_i}|^{p_i} \Big|dx \notag \\
    &\leq \overline{c } \sum_{i=1}^n \int_{B_R}a_i(x)\,\left(  |u^\eta_{x_i}| +|u_{x_i}| \right)^{p_i-1} \,  |u^\eta_{x_i}- u_{x_i}| dx \notag \\
    &\leq \sum_{i=1}^n  c(\Lambda, p_i) \,\int_{B_R} \left(  |u^\eta_{x_i}| +|u_{x_i}| \right)^{p_i-1} \,  |u^\eta_{x_i}- u_{x_i}| dx  ,
\end{align}
where we used the left inequality in the  Lemma \ref{Duzaar}, with $\xi=|u_{x_i}|$ and $\eta=|u_{x_i}^\eta|$ and $\alpha= p_i$.\\
Since, by property of mollification, it holds 
$$u^\eta_{x_i} \to u_{x_i}  \text{ strongly in }L^{p_i},$$
 passing to the limit as $\eta \to 0^+$ in \eqref{convF}, we have
\begin{equation}\label{IntfDuni}
    \int_{B_R}f(x, Du^\eta) \to  \int_{B_R}f(x,Du) .
    \end{equation}
 By virtue of \eqref{IntfDuni}, we observe that the right-hand side of \eqref{stimaconv2} is uniformly bounded w.r.t. $\eta$. Therefore there exists $v$ such that, up to a subsequence,
    $$v^\eta \rightharpoonup v \text{  weakly in  } W^{1, \textbf{p}+2}(B_R, \mathbb{R}^N) \, \text{ as  } \eta \to 0.$$
Once again, by the weak lower semicontinuity of the norm, we derive
\begin{align}\label{stimaconv3}
    \sum_{i=1}^{n} \left(  \int_{B_{R/4}}|v_{x_i}|^{p_i +2 } dx \right) & \leq \liminf_{\eta}\sum_{i=1}^{n} \left(  \int_{B_{R/4}}|(v^\eta)_{x_i}|^{p_i +2 } dx \right) \notag \\
    & \leq \lim_{\eta}  c  \left( \sum_{i=1}^{n} \int_{B_R}f(x,D u^\eta)  dx+ \Vert g \Vert_{L^r (B_{R})} \right)^\sigma  \notag \\
    &= c  \left( \sum_{i=1}^{n} \int_{B_R}f(x,D u)  dx+ \Vert g \Vert_{L^r (B_{R})} \right)^\sigma. 
\end{align}
In order to conclude the proof, it suffices to show that $u=v$ a.e. in $B_R$.\\
By the weak lower semicontinuity of the functionals  $\mathcal{F}$ and $\mathcal{F}^k$ in $W^{1, \textbf{p}}(B_R)$ , the weak convergence of $v^\eta \rightharpoonup v$ in $W^{1, \textbf{p}+2}(B_R)$, \eqref{fktoF} and $v^\eta_k \rightharpoonup v^\eta$ in $W^{1, \textbf{p}+2}(B_R)$, we get
\begin{align}
    \int_{B_R}f(x, Dv) dx & \leq \liminf_{\eta} \int_{B_R} f(x, Dv^\eta) dx \leq \liminf_{\eta} \liminf_{k} \int_{B_R} f^k(x, Dv^\eta) dx \notag \\   
    & \leq \liminf_{\eta}\liminf_{k} \int_{B_R} f^k(x, Dv^\eta_k) dx \notag \\
     & \leq \liminf_{\eta}\liminf_{k} \int_{B_R} \left( f^k(x, Dv^\eta_k) )+ \varepsilon(1+ |Du^\eta_{k, \varepsilon}|^2)^{\frac{p_n}{2}}   \right)dx, \notag 
    \end{align}
   where the last inequality is trivial, by the non-negativity of the last term.\\
    Thanks to the weak convergence of $u^\eta_{k, \varepsilon} \rightharpoonup v^\eta_k$ in $W^{1, \textbf{p}+2}(B_R)$ and again the weak lower semicontinuity of functional $\mathcal{F}^k_\varepsilon$, we have 
   \begin{align}\label{pallino}
    \int_{B_R}f(x, Dv) dx & \leq \liminf_{\eta}\liminf_{k}\liminf_{\varepsilon} \int_{B_R} \left( f^k(x, Du^\eta_{k, \varepsilon})+ \varepsilon(1+ |Du^\eta_{k, \varepsilon}|^2)^{\frac{p_n}{2}}  \right) dx \notag \\
     & \leq \liminf_{\eta}\liminf_{k}\liminf_{\varepsilon} \int_{B_R} \left( f^k(x, Du^\eta)+ \varepsilon(1+ |Du^\eta|^2)^{\frac{p_n}{2}} \right)dx \notag \\
      & = \liminf_{\eta}\liminf_{k} \int_{B_R} f^k(x, Du^\eta) dx , 
      \end{align} 
      where we used the minimility of $u^\eta_{k, \varepsilon}$ and \eqref{fkappalim}.
      By \eqref{pallino},  \eqref{fktoF} and \eqref{IntfDuni}, we derive
      \begin{align}
    \int_{B_R}f(x, Dv) dx  \leq \liminf_{\eta} \int_{B_R} f(x, Du^\eta) dx  \leq  \int_{B_R} f(x, Du) dx,
\end{align}
and then, by minimality of  $u$,  we infer
\begin{equation*}
     \int_{B_R} f(x, Dv) dx =\int_{B_R} f(x, Du) dx.
\end{equation*}
By the strict convexity of $\xi \to f(x, \xi)$, we deduce that $u=v$. Then $u \in W^{1, \textbf{p}+2}$ and so we can argue as in the proof of Theorem \ref{AppThm}, thus obtaining 
\begin{align}
       \sum_{i=1}^{n}   \left( \int_{B_{\frac{R}{4}}} |u_{x_i x_j}|^2|u_{x_i}|^{p_i-2} dx \right) & \leq c  \left( \sum_{i=1}^{n} \Vert u_{x_i} \Vert_{L^{p_i}(B_R)} + \Vert g \Vert_{L^r (B_{R})} \right)^\sigma ,
\end{align}
and this conclude the proof.

\vskip0.5cm

\noindent {{\bf Acknowledgements.}The author is a member of the Gruppo Nazionale per l'Analisi Matematica, la Probabilità e le loro Applicazioni (GNAMPA) of the Istituto Nazionale di Alta Matematica (INdAM). She has also been supported through the project: Sustainable Mobility Center (Centro Nazionale per la Mobilità Sostenibile – CNMS) - SPOKE 10. This manuscript reflects only the author's views and opinions, and the Ministry cannot be considered responsible for them.

\end{document}